\documentclass{amsart}
\usepackage{amsthm}
\usepackage{amssymb, latexsym}
\usepackage{enumerate}
\usepackage{mathtools,xcolor}
\usepackage{dsfont}
\numberwithin{equation}{section}

\newcommand{\HHb}{\boldsymbol{\dot{\mathcal{H}}}_{rad}}
\newcommand{\HH}{\dot{\mathcal{H}}_{rad}}
\newcommand{\HHR}{\boldsymbol{\dot{\mathcal{H}}}_{R}}
\newcommand{\Hb}{\boldsymbol{\dot{H}}_{rad}}
\newcommand{\HR}{\boldsymbol{\dot{H}}_{R}}
\newcommand{\RR}{\mathbb{R}}
\newcommand{\NN}{\mathbb{N}}

\newcommand{\Sb}{\mathbb{S}}
\newcommand{\Lb}{\mathbf{L}}

\newcommand{\indic}{1\!\!1}

\newcommand{\JJJ}{\mathcal{J}}
\newcommand{\NNN}{\mathcal{N}}

\newcommand{\ZZZ}{\mathcal{Z}}
\newcommand{\eps}{\varepsilon}
\newcommand{\tlambda}{\tilde{\lambda}}
\newcommand{\tE}{\widetilde{E}}
\newcommand{\tQ}{\widetilde{Q}}
\newcommand{\tC}{\widetilde{C}}

\newtheorem{prop}{Proposition}[section]
\newtheorem{corol}[prop]{Corollary}
\newtheorem{thm}[prop]{Theorem}
\newtheorem{lem}[prop]{Lemma}
\newtheorem{claim}[prop]{Claim}

\theoremstyle{definition}

\newtheorem{rem}[prop]{Remark}
\newtheorem{definition}[prop]{Definition}
\newtheorem{exemple}[prop]{Example}
\newtheorem{extthm}{Theorem}

\theoremstyle{remark}
\newtheorem{step}{Step}

\DeclareMathOperator{\wlim}{w-lim}

\newcommand{\lf}{\left}
\newcommand{\rg}{\right}

\begin{document}

\begin{abstract}
 This work concerns a general system of energy-critical wave equations in the Minkowski space of dimension $1+3$. The wave equations are coupled by the nonlinearities, which are homogeneous of degree 5. 
 
 We prove that any radial solution of the system 
can be written asymptotically as a sum of rescaled stationary solutions plus a radiation term, along any sequence of times for which the solution is bounded in the energy space. With an additional structural assumption on the nonlinearity, we prove a continuous in time resolution result for radial solutions.

The proof of the sequential resolution uses the channel of energy method, as in the scalar case treated in \cite{DuKeMe13}. The proof of the continous in time resolution is based on new compactness and localization arguments. 
\end{abstract}

\title[Radial solutions of wave systems] {Classification of radial solutions of energy-critical wave systems}
\author{Thomas Duyckaerts${}^1$}
\address{LAGA (UMR 7539), Universit\'e Sorbonne Paris Nord, and \'Ecole Normale Sup\'erieure}
\email{duyckaer@math.univ-paris13.fr}
\author{Tristan Roy${}^2$}
\address{American University of Beirut, Department of Mathematics}
\email{tr14@aub.edu.lb}

\thanks{$^1$LAGA (UMR 7539), Universit\'e Sorbonne Paris Nord, and \'Ecole Normale Sup\'erieure. Partially supported by the program C\`edre, project number 4456XJ}

\thanks{$^2$American University of Beirut, Department of Mathematics. Partially supported by the program C\`edre, project number 4456XJ}



\maketitle
\tableofcontents

\section{Introduction and main results}
In this paper, we are interested with the following system of wave equations on $\RR^3$:
\begin{equation}
 \label{eq:NLW}
 \left\{ \begin{aligned}
\partial_{tt}u-\triangle u &= f(u)\\
\vec{u}_{\restriction t=t_0} & =(u_0,u_1) \in \HHb ,
         \end{aligned}
\right.
\end{equation}
where the unknown function $u$ takes values in $ \RR^m $, $ m \geq 1 $,
$$ \HHb =\Hb \times \Lb^2_{rad}, \quad \Hb = \dot{H}^1_{rad}(\RR^3,\RR^m), \quad  \Lb^2_{rad} = L^2_{rad}(\RR^3,\RR^m),$$
by definition
$V_{rad}$ is the subspace of elements of $V$ depending on $|x|=r$, $f$ is a nonlinearity such that
\begin{equation}
 \label{H1}
 \tag{A0}
f\in  \mathcal{C}^2( \mathbb{R}^{m},\mathbb{R}^{m}),
\quad\text{homogeneous of degree }5,
\end{equation}
that is $f(\lambda x ) = \lambda^{5} f(x)$ for all $ (\lambda, x) \in \mathbb{R}^{+} \times \mathbb{R}^{m} $. In several part of the article, we will also assume that $f$ is of potential-type, i.e
\begin{equation}
\label{H0}
\tag{A1}
f(u)=\nabla_u F(u),\quad F\in \mathcal{C}^3(\RR^m,\RR),\quad \text{homogeneous of degree }6.
\end{equation}
Note that \eqref{H0} implies \eqref{H1}.
By the same proof as in the scalar case (see e.g. the references in the introductions of \cite{Tao06DPDE}, \cite{KeMe08} and subsections \ref{sub:Cauchy1} and \ref{sub:Cauchy2} below), if $f$ is as in \eqref{H1}, the equation \eqref{eq:NLW} is locally well-posed in $\HHb$. The equation is invariant by scaling: if $u(t,x)$ is a solution, so is $\lambda^{1/2} u(\lambda t,\lambda x)$. Furthermore, assuming \eqref{H0}, the solution $u$ of (\ref{eq:NLW}) satisfies the conservation law $E(\vec{u}(t)) = E(\vec{u}(0))$ for all time $t$  in its domain of existence $(T_-(u),T_+(u))$, with
\begin{equation}
\nonumber
E(\vec{u}(t)) := \frac{1}{2} \int_{\mathbb{R}^{3}} |\nabla u(t)|^{2} \,dx +\frac 12\int_{\RR^3}|\partial_tu(t)|^2dx- \int_{\mathbb{R}^{3}} F(u(t)) \,dx \cdot
\end{equation}
Hence we denote by $E(u):= E(\vec{u}(t))$ the energy of a solution $u$. Let
\begin{equation}
\label{DefZ}
\mathcal{Z}:=  \left\{ Z \in \Hb  \setminus \{ 0 \}: - \triangle Z  = f(Z) \right\}.
\end{equation}
If $Z \in \mathcal{Z}$ then $E(Z)$ denotes the energy of $Z$, i.e
\begin{equation}
\nonumber
E(Z) := \frac{1}{2} \int_{\mathbb{R}^{3}} |\nabla Z|^{2} \,dx - \int_{\mathbb{R}^{3}} F(Z) \,dx \cdot
\end{equation}

Our aim is to describe the asymptotic dynamics of the solution of \eqref{eq:NLW} as $t\to T_+(u)$.
We first recall this dynamics in the scalar case $m=1$, where it is well understood.

\subsection{Scalar case}

Let $v_L(t)=S_L(t)(u_0,u_1)$ be defined by
\begin{equation}
\label{defSL}
S_L(t)(u_0,u_1)=\cos\left(t\sqrt{-\triangle}\right)u_0+\frac{\sin\left(t\sqrt{-\triangle}\right)}{\sqrt{-\triangle}}u_1.
\end{equation}
It is the solution of the linear wave equation on $\RR^3$:
\begin{equation}
 \label{eq:LW}
 \partial_{tt}v_L-\triangle v_L=0,
\end{equation}
with initial data $(u_0,u_1)\in \HHb$ at $t=0$. We will denote $\vec{S}_L(t)=(v_L(t),\partial_tv_L(t))$.

In the case $m=1$, two typical nonlinearities are given by $f(u)=-u^5$ and $f(u)=u^5$.

For the defocusing nonlinearity $f(u)=-u^5$, the set $\mathcal{Z}$ is reduced to the constant $0$ solution, and every solution $u$ of \eqref{eq:NLW} is global and in $L^5(\RR,L^{10})$ (see e.g. \cite{GiSoVe92}), which implies easily that it scatters to a linear solution: there exists $(v_0,v_1)\in \HH$ such that
$$\lim_{t\to\infty}\|\vec{u}(t)-\vec{S}_L(t)(v_0,v_1)\|_{\HH}=0,$$
where the notation $\vec{u}$ denotes the pair $(u,\partial_tu)$.

In the focusing case $f(u)=u^5$, the equation has an explicit stationary solution
\begin{equation}
\label{def_W}
 W(r)=\left( 1+|x|^2/3 \right)^{-1/2}.
\end{equation}
Denoting by $W_{(\lambda)}(x)=\lambda^{-1/2}W(\lambda^{-1}x)$, one can check, using e.g. ODE arguments, that $\mathcal{Z}=\{\pm W_{(\lambda)},\lambda>0\}$.

In \cite{DuKeMe13}, the following ``resolution into stationary solutions'' was proved:
\begin{extthm}
\label{T:DKM}
Assume $m=1$ and $f(u)=u^5$.
Let $u$ be a solution of \eqref{eq:NLW} such that $T_+(u)=+\infty$ or
\begin{equation}
\label{liminf_scal}
T_+(u)<\infty\text{ and }\liminf_{t\to T_+(u)}\|\vec{u}\|_{\HH}<\infty.
\end{equation}
Then there exist $J\geq 0$, $(v_0,v_1)\in \HH$, $(\lambda_{1}(t),\ldots,\lambda_{J}(t))\in \mathcal{C}^0([t_0,T_+(u)),(0,\infty)^{J})$, $(\iota_j)_{j}\in \{\pm 1\}^{J}$ 
such that
$$ \lambda_1(t)\ll \ldots \ll \lambda_J(t),\quad t\to T_+(u),$$
and
\begin{itemize}
 \item if $T_+(u)<\infty$, then $J\geq 1$, $\lambda_{J}(t)\ll T_+-t$ as $t\to T_+(u)$ and
$$ \vec{u}(t)=(v_0,v_1)+\sum_{j=1}^{J} \left(\iota_j W_{(\lambda_j(t))},0\right)+o(1)\text{ in } \HH\text{ as }t\to T_+(u);$$
\item if $T_+(u)=+\infty$, then $\lambda_J(t)\ll t$ as $t\to+\infty$ (if $J\geq 1$), and
$$ \vec{u}(t)=\vec{S}_L(t)(v_0,v_1)+\sum_{j=1}^{J} \left(\iota_j W_{(\lambda_j(t))},0\right)+o(1)\text{ in } \HH\text{ as }t\to +\infty.$$
\end{itemize}
\end{extthm}
Solutions that satisfy \eqref{liminf_scal} are called \emph{Type II Blow-up} solutions. Examples of solutions satisfying the conclusion of Theorem \ref{T:DKM} are known for $J=0$, $T_+(u)=+\infty$ (scattering solutions) and for $J=1$, $T_+=\infty$ (see \cite{KrSc07} and \cite{DoKr13}) and $J=1$, $T_+<\infty$ (see \cite{KrScTa09}).

To complete the scalar case, let us mention that Theorem \ref{T:continuous_resolution} below implies the analog of Theorem \ref{T:DKM} for a
 general homogeneous nonlinearity ($f(u)=c_+u^5$ for $u \geq 0$, $f(u)=c_-u^5$ for  $u \leq 0$, where $c_+,c_-$ are real constants).
\subsection{Convergence for a sequence of times}
The goal of this article is to investigate generalizations of Theorem \ref{T:DKM} in the vector-valued case. We will prove an analog of Theorem \ref{T:DKM} when the nonlinearity is of the form \eqref{H0}, with an additional assumption on the energies of the stationary solutions. We first state two weaker statements that are valid for nonlinearities satisfying the weaker assumption \eqref{H1}.
\begin{prop}
\label{P:sequence_resolution}
Assume \eqref{H1}.
Let $u$ be a solution of \eqref{eq:NLW} such that
\begin{equation}
\lim \inf_{t \rightarrow T_+(u)} \left\| \vec{u}(t) \right\|_{\HHb} < \infty .
\label{liminf}
\end{equation}
Then there exists $(v_0,v_1)\in \HHb$ such that
\begin{itemize}
 \item If $T_+(u)=+\infty$, then $$\forall A\in \RR,\quad \lim_{t\to\infty}\int_{|x|>|t|+A}\left(|\partial_t(u-v_L)(t)|^2+|\nabla (u-v_L)(t)|^2\right)dx=0,$$
 where $v_L(t)=S_L(t)(v_0,v_1)$.
 \item If $T_+(u)<\infty$,
 $$\lim_{t\to T_+(u)} \int_{|x|>T_+(u)-t} \left(|\partial_tu(t,x)-v_1(x)|^2+|\nabla (u(t,x)-v_0)(x)|^2\right)dx.$$
\end{itemize}
\end{prop}
\begin{thm}
\label{T:sequence_resolution}
Assume \eqref{H1}. Let $u$, $(v_0,v_1)$ be as in Proposition \ref{P:sequence_resolution}, and $t_n\to T_+(u)$ be a sequence of times such that the sequence
$\left(\|\vec{u}(t_n)\|_{\HHb}\right)_n$ is bounded. Then there exist a subsequence of $(t_n)_n$ (that we still denote by $(t_n)_n$),
an integer $J \in \mathbb{N}$, positive numbers $ (\lambda_{j,n})_{1 \leq j \leq J}$, and
$( Q_{j} )_{1 \leq j \leq J} \subset \mathcal{Z}$ such that
\begin{equation}
\lambda_{1,n} \ll \lambda_{2,n} \ll ... \ll \lambda_{J,n}, \; \text{and}
\label{Eqn:OrthCondLambda}
\end{equation}
\begin{itemize}
\item if $T_+(u)=+\infty$, then $\lambda_{J,n}\ll t_n$ (if $J\geq 1$), and, as $n\to\infty$,
\begin{equation}
\label{expansion_seq}
\vec{u}(t_{n}) =
 \vec{S}_L(t_{n})(v_0,v_1) + \sum \limits_{j=1}^{J} \frac{1}{\lambda_{j,n}^{\frac{1}{2}}}  \left( Q_{j} \left( \frac{x}{\lambda_{j,n}} \right) , 0 \right)
+ o(1) \text{ in }\HHb;
\end{equation}
\item if $T_+(u)<\infty$, then $J\geq 1$, $\lambda_{J,n}\ll T_+-t_n$ and, as $n\to\infty$,
\begin{equation}
\label{expansion_seq'}
\vec{u}(t_{n}) =
 (v_0,v_1) + \sum \limits_{j=1}^{J} \frac{1}{\lambda_{j,n}^{\frac{1}{2}}}  \left( Q_{j} \left( \frac{x}{\lambda_{j,n}} \right) , 0 \right)
+ o(1)\text{ in }\HHb.
\end{equation}
\end{itemize}
\end{thm}
Let us mention that if $f$ satisfies \eqref{H0}, i.e. it is a potential-type nonlinearity then \eqref{liminf} holds. Indeed in this case any global solution is uniformly bounded in the energy space $\HHb$:
\begin{prop}
 \label{P:bounded}
Assume that \eqref{H0} holds. Let $u$ be a solution of \eqref{eq:NLW} such that $T_+(u)=\infty$. Then
$$\limsup_{t\to\infty}\|\vec{u}(t)\|_{\HHb}^2\leq 3E(u).$$
\end{prop}
\begin{rem}
It is an open question if the conclusion of Proposition \ref{P:bounded} still holds for general nonlinearities $f$ satisfying \eqref{H1}, but not \eqref{H0}. See also the discussion in the end of Section \ref{S:higherD} concerning analogs of the system \eqref{eq:NLW} in higher dimensions.
\end{rem}
Theorem \ref{T:sequence_resolution} asserts that for any sequence of times $t_n\to T_+(u)$, there exists a subsequence such that an expansion of the form \eqref{expansion_seq} or \eqref{expansion_seq'} holds. This is stronger than the existence of a single sequence $t_n\to T_+(u)$ having this property (as for example in the main theorem of \cite{DuJiKeMe17a}). As an immediate corollary, we obtain:
\begin{corol}
\label{cor:sequence}
 Assume \eqref{H1}, and $\ZZZ=\emptyset$. Let $u$ be a solution of \eqref{eq:NLW} such that
 $$ \liminf_{t\to T_+(u)}\|\vec{u}(t)\|_{\HHb}<\infty.$$
 Then $T_+(u)=+\infty$ and $u$ scatters for positive times to a linear solution.
\end{corol}
Indeed, the fact that $J\geq 1$ if $T_+(u)<\infty$ (in the conclusion of Theorem \ref{T:sequence_resolution})  shows that we must have $T_+(u)=+\infty$. Next Theorem \ref{T:sequence_resolution} and a contradiction argument show that for any sequence $t_n\to\infty$ such that $\left(\|\vec{u}(t_n)\|_{\HHb}\right)_n$ is bounded, $\lim_{n \to \infty}\|(\vec{u}-\vec{v}_L)(t_n)\|_{\HHb}=0$, where $\vec{v}_L(t)=S_L(t)(v_0,v_1)$. We next claim that for any sequence $t_{n} \to \infty$ the sequence
 $\left( \|\vec{u}(t_n)\|_{\HHb}\right)_n$ is bounded. Indeed assume that this is not true. Then using also the fact that there exists at least one sequence $\bar{t}_{n} \to \infty $ such that $\left(\|\vec{u}(\bar{t}_n)\|_{\HHb}\right)_n$ is bounded  we see from the intermediate value theorem that there exists a sequence $\bar{t}^{'}_{n} \to 
 \infty $ such that $ \|\vec{u}(\bar{t}^{'}_n)\|_{\HHb} =  \left\| (v_{0},v_{1}) \right\|_{\HHb} + 1 $, which contradicts  $ \lim_{n \to \infty} \| (\vec{u}-\vec{v}_L) (\bar{t}^{'}_n) \|_{\HHb} = 0 $ and the identity  $ \left\| \vec{v}_L (\bar{t}^{'}_n) \right\|_{\HHb} =
 \left\| (v_{0},v_{1}) \right\|_{\HHb} $. This concludes the proof of Corollary \ref{cor:sequence}.


\subsection{Resolution for all times}
In the case where $\ZZZ$ is not empty, the assumption \eqref{H0}, and even the stronger assumption \eqref{H1} do not seem sufficient to prove a continuous in time resolution i.e. the analogue of Theorem \ref{T:DKM} for systems. We were however able to prove such a result with the following additional assumptions \eqref{Eqn:AssFinite} and \eqref{Eqn:AssEj} on $\mathcal{Z}$:
\begin{equation}
\tag{A2}
\text{The set }E(\ZZZ)=\{E(Z),\; Z\in \ZZZ\}\text{ is finite.}
\label{Eqn:AssFinite}
\end{equation}
Assuming \eqref{Eqn:AssFinite}, we will denote \begin{equation}
\label{defEj}
E(\ZZZ)=\{E_1,\ldots, E_p\},\quad 0<E_1<\ldots<E_p,
\end{equation}
and we will make the following assumption:
\begin{multline}
\tag{A3}
\label{Eqn:AssEj}
\forall j\in \{1,\ldots,m\},\, \forall (\alpha_k)_{k}\in \NN^{m},\\ E_j=\sum_{k=1}^m \alpha_k E_k
\Longrightarrow \left(\alpha_j=1\text{ and }\forall k\neq j,\; \alpha_k=0\right).
\end{multline}

\begin{rem}
The fact that $E_1>0$ follows from
Proposition \ref{Prop:PropZ} (see in particular \eqref{I:Z1} and \eqref{I:Z2} in the statement of this Proposition).
%
\end{rem}
We write below our theorem on resolution into stationary solutions:
 \begin{thm}
 \label{T:continuous_resolution}
Assume \eqref{H0}, \eqref{Eqn:AssFinite} and \eqref{Eqn:AssEj}. There exists a subset $K$ of $\Sigma$, compact for the strong topology of $\Hb$, with the following property. Let $u$ be a solution of \eqref{eq:NLW} such that $T_{+}(u)= \infty$ or \eqref{liminf} holds.
Then there exist an integer $J\geq 0$, and, for $j\in \{1,\ldots,J\}$, real-valued functions $\lambda_j\in \mathcal{C}^0([t_0,T_+(u)))$ with $\lambda_j(t)>0$,  and continuous functions $Q_j:[t_0,T_+(u))\to K$, such that
$$\lambda_{1}(t) \ll \lambda_{2}(t) \ll... \ll \lambda_{J}(t),\quad t \to T_+(u)$$
and, letting $(v_0,v_1)$ be as in Theorem \ref{T:sequence_resolution},
\begin{itemize}
 \item If $T_+(u)=+\infty$ then, as $t\to\infty$,   $\lambda_J(t)\ll t$  (if $J\geq 1$), and
 \begin{equation*}
\vec{u}(t) =S_L(t)(v_0,v_1)+ \sum_{j=1}^{J} \left(\frac{1}{\lambda_{j}^{\frac{1}{2}}(t)} Q_{j} \left( t, \frac{\cdot}{\lambda_{j}(t)} \right),0\right) + o(1)\quad \text{in }\HHb.
\end{equation*}
\item If $T_+(u)<\infty$, then $J\geq 1$ and, as $t\to T_+(u)$, $\lambda_J(t)\ll T_+-t$ and
\begin{equation*}
\vec{u}(t) =(v_0,v_1)+ \sum_{j=1}^{J} \left(\frac{1}{\lambda_{j}^{\frac{1}{2}}(t)} Q_{j} \left( t, \frac{\cdot}{\lambda_{j}(t)} \right),0\right) + o(1)\quad \text{in }\HHb.
\end{equation*}
\end{itemize}
\end{thm}
\begin{rem}
The proof of Theorem \ref{T:continuous_resolution} shows that one can take:
 \begin{equation*}
K = \left\{ Q \in \ZZZ,\quad \int_{|x| \leq 1}  |\nabla Q|^{2} \, dx =  \frac{1}{2}\int |\nabla Q|^2\,dx.\right\}.
\end{equation*}
\end{rem}

\begin{rem}
\label{R:energy}
Consider $f$, $u$ satisfying the assumptions of the theorem with $T_+(u)=+\infty$.
Observe that the conservation of energy, the condition  $ \lambda_{1}(t) \ll ... \ll \lambda_{J}(t) \ll t$, and similar arguments as the proof of Lemma
\ref{Lem:CompactnessK} (in particular those from (\ref{Eqn:EstFDecoupl}) to the end of the proof)  show that
$  \sum \limits_{j=1}^{J} E(Q_{j}) = E(u) - E(\vec{v}_{L}(t))  + o(1)$  for $ t \gg 1 $ with
\begin{equation}
\begin{array}{l}
E(\vec{v}_{L}(t)) := \frac{1}{2} \int_{\mathbb{R}^{3}} | \nabla  v_{L}(t)|^{2} \,dx
+  \frac{1}{2} \int_{\mathbb{R}^{3}} | \partial_{t} v_{L}(t)|^{2} \,dx + \int_{\mathbb{R}^{3}} F(v_{L}(t)) \,dx .
\end{array}
\nonumber
\end{equation}
Recall that $ \| v_{L} (t) \|_{L^{6}} \rightarrow 0 $ as $ t \rightarrow \pm \infty $ \footnote{this is clear from the dispersive estimates for smooth data; if the data belongs to $\dot{H}^{1}(\mathbb{R}^{3})$  then this follows from an approximation argument in $\dot{H}^{1}(\mathbb{R}^{3})$ combined with the embedding
$\dot{H}^{1}(\mathbb{R}^{3}) \subset L^{6}(\mathbb{R}^{3})$.}. Hence, using also that $|F(u)|\lesssim |u|^6$ and the fact that the kinetic part of $E(\vec{v}_{L}(t))$ is conserved, we get
\begin{equation}
\begin{array}{l}
\sum \limits_{j=1}^{J} E(Q_{j}) = E(u) - \frac{1}{2} \int_{\mathbb{R}^{3}} | \nabla  v_{0}|^{2} \,dx  - \frac{1}{2} \int_{\mathbb{R}^{3}} |v_{1}|^{2} \,dx,
\end{array}
\nonumber
\end{equation}
with $(v_{0},v_{1}):=\vec{v}_{L}(0)$.
\end{rem}
%


\subsection{Examples}
We next give some applications of Theorem \ref{T:continuous_resolution} for specific nonlinearities. In Section \ref{S:stationary}, we will prove (assuming \eqref{H0}):
\begin{equation}
 \label{CNS_emptyZ'}
  \ZZZ\neq \emptyset \iff \exists u\in \RR^m \text{ s.t. } F(u)>0.
 \end{equation}
(see Proposition \ref{P:GroundState}). As a consequence of \eqref{CNS_emptyZ'} and Theorem \ref{T:continuous_resolution}, we have
\begin{corol}
 \label{cor:emptyZ}
 Assume \eqref{H0} and $F(u)\leq 0$ for all $u\in \RR^m$. Then every solution of \eqref{eq:NLW} is global and scatters to a linear solution.
\end{corol}
Corollary \ref{cor:emptyZ} is already known in the case of a defocusing nonlinearity that satisfies:
$$ \exists c>0,\quad \forall u\in \RR^m, \quad -F(u)\geq c  |u|^6 $$
(where $|\cdot|$ is the Euclidean norm in $\RR^m$). In this case, the Morawetz inequality
$$ -\int_0^T \int_{\RR^3} \frac{F\left(u(t,x)\right)}{|x|}dx dt\lesssim E(u_0,u_1),$$
is sufficient to ensure global well-posedness and scattering  with a particularly simple proof in the radial case  (see e.g. \cite{GiSoVe92} and the introductions of \cite{Tao06DPDE} and \cite{Tao16}). Corollary \ref{cor:emptyZ} gives a generalization of this result, in the radial case, to a potential $F$ which is only assumed to be nonnegative.

\medskip

When there exists $u\in \RR^m$ such that $F(u)>0$, one can prove that there exists $\theta\in \RR^m\setminus \{0\}$ such that $f(\theta)=\theta$, and $\theta W$ is a stationary solution of \eqref{eq:NLW}. If furthermore $f$ is odd, the system \eqref{eq:NLW} admits solutions of the form $\theta \varphi$, where $\varphi$ is solution of the scalar focusing equation $\partial_{tt}\varphi-\triangle \varphi=\varphi^5$. As a consequence, there exist type II blow-up solutions, as the one constructed in \cite{KrScTa09}, and global, non-scattering solutions that are not stationary, as in \cite{KrSc07} and \cite{DoKr13}.

In some cases, we can compute the set $\ZZZ$, giving a more explicit version of Theorem \ref{T:continuous_resolution}. As a typical example, consider the focusing nonlinearity given by $F(u)=-\frac 16 |u|^6$, $f(u)=-|u|^4u$, where $|\cdot|$ is the Euclidean norm on $\RR^m$. In this case we have (see Example \ref{Ex:Euclidean})
$$\ZZZ=\left\{\omega W_{(\lambda)}\;:\; \omega\in \RR^m,\; |\omega|=1,\; \lambda>0\right\}.$$
As a consequence, the elements of $\ZZZ$ all have the same energy $E(W,0)$, and the assumption \eqref{Eqn:AssEj} is satisfied. Theorem \ref{T:continuous_resolution} implies:
\begin{thm}
 \label{T:continuous_resolution'}
Assume that \eqref{H0} is satisfied with $F(u)=-\frac 16|u|^6$.
 Let $u$ be a solution of \eqref{eq:NLW} such that $T_{+}(u)= +\infty$.
Then there exist an integer $J\geq 0$, and, for $j\in \{1,\ldots,J\}$, real-valued functions $\lambda_j\in \mathcal{C}^0([t_0,+\infty))$ with $\lambda_j(t)>0$,  and continuous functions $\omega_j:[t_0,+\infty)\to S^{m-1}$, such that as $t \rightarrow +\infty$, we have
\begin{equation}
\label{SR12}
\vec{u}(t) = S_L(t)(v_0,v_1)+ \sum_{j=1}^{J} \left(\frac{\omega_j(t)}{\lambda_{j}^{\frac{1}{2}}(t)} W\left( \frac{\cdot}{\lambda_{j}(t)} \right),0\right) + o(1)\quad \text{in }\HHb
\end{equation}
with $\lambda_{1}(t) \ll \lambda_{2}(t) \ll... \ll \lambda_{J}(t) \ll t$ as $t \to +\infty$.
\end{thm}
(With the usual variant in the case where $T_+(u)<\infty$ and \eqref{liminf} holds).
Note that the preceding example includes the case of complex valued solutions of the equation $\partial_{tt}u-\triangle u=|u|^4u$, where $|u|$ is the modulus. In this case, by a standard lifting lemma, one can take $\omega_j(t)=e^{i\theta_j(t)}$, where $\theta$ is a continuous, real-valued function.

There are other examples where the set $\ZZZ$ can be computed explicitly, see Section \ref{S:stationary}. 
The complete description of the set $\ZZZ$ for general nonlinearity satisfying \eqref{H0} or the weaker assumption \eqref{H1} seems a difficult question. One interesting open problem is the existence of nonlinearities $f$ satisfying \eqref{H0} such that $\ZZZ$ admits some elements which are not linear combinations (with coefficients in $\RR^m$) of rescaled $W$.

\subsection{Comments on the proofs}
The proof of Theorem \ref{T:sequence_resolution} is based on the \emph{channels of energy} method, which was introduced in  \cite{DuKeMe11a}, and extended in of \cite{DuKeMe13} to prove Theorem \ref{T:DKM}. The key point is to classify the so-called \emph{non-radiative solutions} of \eqref{eq:NLW}, that are solutions defined at least on the set $|x|>|t|$, and such that
\begin{equation}
 \label{nonrad}
 \sum_{\pm}
\lim_{t\to\pm\infty} \int_{|x|>|t|} |\nabla_{t,x}u(t,x)|^2dx=0.
\end{equation}
The elements of $\ZZZ$ are non-radiative. In Section \ref{S:rigidity} we prove the following rigidity theorem: any nonzero solution of \eqref{eq:NLW} such that \eqref{nonrad} holds is in $\ZZZ$. The proof is close to the one of the analogous result in \cite{DuKeMe13}. Using concentration compactness argument (based on the profile decomposition of Bahouri and G\'erard \cite{BaGe99}) one can deduce the resolution for sequence of times (Theorem \ref{T:sequence_resolution}) from the rigidity theorem (see Section \ref{S:sequence_resolution}).

For the scalar focusing wave equation, the full resolution (Theorem \ref{T:DKM}) follows quickly from its sequential version (see \cite{DuKeMe13}). The proof of Theorem \ref{T:continuous_resolution} assuming Theorem \ref{T:sequence_resolution} in the case of systems is much more intricate. A key point is to prove, using assumptions \eqref{Eqn:AssFinite} and \eqref{Eqn:AssEj}, that the number $J$ and the energies of the stationary states $Q_j$ in any expansion of the form \eqref{expansion_seq} are fixed (see Subsection \ref{sub:SameOrder}). Note that in the scalar case, this crucial property follows directly from the conservation of the energy and the fact that all nonzero stationary solutions have the same energy.

\subsection{Outline of the paper}

We start with some notations and preliminaries on well-posedness and profile decomposition, see Section \ref{S:notations}. We develop in particular, using finite speed of propagation, a well-posedness theory on domains of influences and a nonlinear profile decomposition outside wave cones for the system \eqref{eq:NLW} in the spirit of Subsections 2.3 and 2.4 of \cite{DuKeMe21a}.

In Section \ref{S:stationary}, we study the set $\ZZZ$, in the cases where assumption \eqref{H1} or the stronger \eqref{H0} is satisfied. We prove the rigidity theorem in Section \ref{S:rigidity}, the sequential resolution in Section \ref{S:sequence_resolution}.
The next two sections are devoted to potential-like nonlinearities, i.e. nonlinearities satisfying \eqref{H0}. We first prove Proposition \ref{P:bounded}, i.e. that when \eqref{H0} is satisfied, global solutions are bounded in $\HHb$ along a sequence of times (see Section \ref{S:bounded}). Section \ref{S:continuous_resolution} concerns the resolution for all times.

We will only detail the proofs of Theorems \ref{T:sequence_resolution} and \ref{T:continuous_resolution} when $T_+(u)=+\infty$. The proofs when $T_+(u)$ is finite are very similar and are omitted. We refer to Section 4 of \cite{DuKeMe13} for a sketch of Theorem \ref{T:DKM} in the finite time blow-up case, which can be easily combined with the arguments of Sections  \ref{S:sequence_resolution} and \ref{S:continuous_resolution} to obtain Theorems
\ref{T:sequence_resolution} and \ref{T:continuous_resolution} in this case.

\medskip

We conclude the introduction by giving several references on systems of wave equations that are related to \eqref{eq:NLW}. The articles \cite{Azaiez15,AzaiezZaag17} concern the description of blow-up for a complex wave equation in space dimension $1$. In the work \cite{Tao16} a radial blow-up solution is constructed for defocusing wave system with a potential-like supercritical nonlinearity in space dimension $3$ (i.e. equation \eqref{eq:NLW} where in the assumption \eqref{H0}, $F(u)$ is negative, and homogeneous of degree larger than $5$ outside the origin). We are not aware of any works on systems of wave equations with energy-critical nonlinearity as the one studied in this article. In Section \ref{S:stationary} we study the set of solutions of the elliptic system $-\triangle u=f(u)$ with $u\in \Hb$. This type of systems have been studied in several paper see e.g. \cite{QuittnerSouplet12} and references therein. We stress out however that the questions that are considered in these works (e.g. existence or non-existence of positive solutions) are quite distinct that the ones that are considered in Section \ref{S:stationary}.

\section{Notations and preliminaries}
\label{S:notations}

\subsection{Notations}
\label{sub:notations}
Let $n\geq 1$, and $\Omega$ an open subset of $\mathbb{R}^n$. Let $V$ be a normed vector space of distributions on $\mathbb{R}^n$. We will denote by $V(\Omega)$ the space of restrictions of functions of $V$ to $\Omega$, endowed with the norm
$$ \|\varphi\|_{V(\Omega)}=\inf_{\underline{\varphi}} \|\underline{\varphi}\|_V,$$
where the infimum is taken over all extensions $\underline{\varphi}$ of $\varphi$ to $\mathbb{R}^n$.

We will denote in bold font the space of functions with values in $\RR^m$.

With this convention, and a slight abuse of notation, for any measurable set $\Gamma \subset \RR^4$, $\Lb^p\Lb^q(\Gamma)$ is the space of measurable functions $u$ on $\Gamma$ such that the extension $\underline{u}$ of $u$ by $0$ to $\RR^4$ is in $\Lb^{p}\Lb^q$. We will denote:
$$\HHR=\HHb(\{|x|>R\})=\HR\times \Lb^2_R,$$
where
$$\HR=\Hb(\{|x|>R\}),\quad \Lb^2_R=\Lb^2_{rad}(\{|x|>R\}),$$
so that
$$\|(u_0,u_1)\|_{\HHR}^2=\int_R^{\infty}|\partial_ru_0(r)|^2r^2dr+\int_R^{\infty}|u_1(r)|^2r^2dr.$$

If $u$ is a function of space and time, we denote $\vec{u}=(u,\partial_tu)$.

In all this section, we assume that $f$ satisfies \eqref{H1}, but not necessarily \eqref{H0}, unless explicitly mentioned.


%
%
%
%
%

\subsection{Strichartz estimates and Cauchy theory}
\label{sub:Cauchy1}
Consider the free wave equation:
\begin{equation}
\label{eq:FW}
 \left\{ \begin{aligned}
\partial_{tt}u-\triangle u&=0\\
\vec{u}_{\restriction t=t_0}&=(u_0,u_1),
         \end{aligned}
\right.
\end{equation}
and the
inhomogeneous wave equation:
\begin{equation}
\label{eq:inh_wave}
 \left\{ \begin{aligned}
\partial_{tt}u-\triangle u&=\varphi\\
\vec{u}_{\restriction t=t_0}&=(u_0,u_1).
         \end{aligned}
\right.
\end{equation}
Let $\varphi\in L^1_{loc}(I,\Lb^2 )$ (where $I$ is an interval with $t_0\in I$), and $(u_0,u_1)\in \HHb$. By definition, the solution $u$ of \eqref{eq:inh_wave} on $I\times \RR^3$ is
$$ u(t)=S_L(t-t_0)(u_0,u_1)+\int_{t_0}^{t} \frac{\sin\left((t-s)\sqrt{-\triangle}\right)}{\sqrt{-\triangle}}\varphi(s)ds,$$
where $S_L(t)$ is defined in \eqref{defSL}. We recall the following Strichartz estimate: if $u$ is a solution of \eqref{eq:inh_wave}, then $u\in \Lb^5\Lb^{10}\left(I\times \RR^3 \right)$, $\vec{u}\in \mathcal{C}^0(I,\HHb)$ and
\begin{equation}
 \label{Strichartz}
\|u\|_{\mathbf{L}^5(I,\Lb^{10})}+\sup_{t\in I} \left\|\vec{u}(t)\right\|_{\HHb}\lesssim \|(u_0,u_1)\|_{\HHb}+\|\varphi\|_{\Lb^1\Lb^2(I\times \RR^3)}.
\end{equation}
By definition, a solution of \eqref{eq:NLW}.
on $I\times \RR^3$ is a solution of \eqref{eq:inh_wave} with $f(u)=\varphi$, such that $u\in L^5_{loc}(I,\Lb^{10})$.
Using \eqref{Strichartz}, the inequality
\begin{equation}
 \label{ineg_f}
 |f(u)-f(v)|\lesssim |u-v|\left(|u|^4+|v|^4\right)
\end{equation}
and a standard fixed point argument, one can prove (exactly as in the scalar case) that for any $(u_0,u_1)\in \HHb$ there is a unique solution of \eqref{eq:NLW}, defined on a maximal time of existence $(T_-,T_+)$, and that satisfies the blow-up criterion
$$ T_+<\infty\Longrightarrow u\notin \Lb^5\Lb^{10}\left([t_0,T_+)\times \RR^3 \right),\quad  |T_-|<\infty\Longrightarrow u\notin
\Lb^5\Lb^{10}\left((T_-,t_0]\times \RR^3 \right).$$
In the next subsection, we will develop, taking into account the finite speed of propagation, a more general Cauchy theory for \eqref{eq:NLW}.
\subsection{Cauchy theory on domains of influence}
\label{sub:Cauchy2}
We develop here a Cauchy Theory for \eqref{eq:NLW} in general space-times domain adapted to finite speed of propagation. This extends
\cite[Section 2.3]{DuKeMe21a}, where a Cauchy theory outside wave cones is developed for the scalar equation.

If $\Gamma$ is an open subset of $\RR\times \RR^3$ and $I$ is an interval, we denote
$$ \Gamma_I=\{(t,x)\in \Gamma,\; t\in I\}.$$
If $(t_0,x_0)\in \RR\times \RR^3$, we denote by $\overline{C}_{t_0,x_0}$ the past closed wave cone $\{(t,x)\in \RR\times \RR^3,\; |x-x_0|\leq t_0-t\}$.
\begin{definition}
Following \cite{Alinhac95Bo},
 we say that an open subset $\Gamma$ of $\RR\times \RR^3$ is a domain of influence if for all $(t_0,x_0)\in \Gamma$, $\overline{C}_{t_0,x_0}\subset \Gamma$.
\end{definition}
\begin{exemple}
 If $t_0\in \RR$, $A\in \RR$, then $\{(t,x),\; |x|\geq t+A\}$ and $\{(t,x),\; |x|\leq A-t\}$ are domains of influence.
\end{exemple}
 The finite speed of propagation for the linear wave equation implies the following property:
\begin{prop}
\label{P:FSP}
Let $\Gamma$ be a domain of influence, $-\infty<t_0<t_1\leq \infty$ and
 $u,\varphi$ and $v,\psi$ be two solutions of \eqref{eq:inh_wave} on $[t_0,t_1)\times \RR^3$. Assume $\varphi=\psi$ almost everywhere on $\Gamma_{[t_0,t_1)}$ and $\vec{u}(t_0)=\vec{v}(t_0)$ almost everywhere on $\Gamma_{\{t_0\}}$. Then $u=v$ almost everywhere on $\Gamma$.
\end{prop}
The proof is the same as in the scalar case, for which we refer to \cite[Remark 2.12]{KeMe08}.

As a consequence of Proposition \ref{P:FSP}, we can generalize the Cauchy theory on the whole space $\RR_{loc}\times \RR^3$ sketched in the previous subsection to domains of influence. To begin with, we define solutions of \eqref{eq:NLW} on domains of influence:
\begin{definition}
Let $\Gamma$ be a domain of influence, $-\infty<t_0<t_1\leq \infty$, $(u_0,u_1)\in \HHb(\Gamma_{\{t_0\}})$ and $\varphi\in \Lb^1_{loc}(\Gamma)$ such that $\varphi\in \Lb^1\Lb^2(\Gamma_{[t_0,T])})$ for all $T<t_1$.
A solution $u$ of \eqref{eq:inh_wave} on $\Gamma_{[t_0,t_1)}$ is the restriction to $\Gamma_{[t_0,t_1)}$ of any solution $\underline{u}$ of
the equation
$$ \partial_{tt}\underline{u}-\triangle \underline{u}=\underline{\varphi},\quad (t,x)\in [t_0,t_1)\times \RR^3,$$
where $\underline{\varphi}\in \Lb^1\Lb^2([t_0,T)\times \RR^3)$ for all $T<t_1$, $\vec{\underline{u}}(t_0)\in \HHb$, $\underline{\varphi}=\varphi$ a.e. on $\Gamma_{[t_0,t_1)}$,
$\vec{\underline{u}}(t_0)=\vec{u}(t_0)$ a.e. on $\Gamma_{\{t_0\}}$.

A solution $u$ of \eqref{eq:NLW} on $\Gamma_{[t_0,t_1)}$ is a solution of \eqref{eq:inh_wave} on $\Gamma_{[t_0,t_1)}$ with $u\in\Lb^5\Lb^{10}(\Gamma_{[t_0,T]})$, for all $T\in (t_0,t_1)$ and $\varphi=f(u)$.
\end{definition}
By finite speed of propagation, the solution $u$ of \eqref{eq:inh_wave} on $\Gamma_{[t_0,t_1)}$ depends only of the values of $\varphi$ on $\Gamma$ and of $(u_0,u_1)$ on $\Gamma_{\{t_0\}}$ and is therefore unique. Also, adapting the usual well-posedness theory and using finite speed of propagation, we obtain:
\begin{prop}
\label{Prop:WP}
Let $\Gamma$ be a domain of influence, $t_0\in \RR$, with $\Gamma_{\{t_0\}}\neq \emptyset$ $(u_0,u_1)\in \HHb(\Gamma_{\{t_0\}})$. Then there exists a maximal time of existence $T_+$ and a unique maximal solution $u$ of \eqref{eq:NLW} defined on $\Gamma_{[t_0,T_+)}$. Furthermore, one has the following blow-up criterion:
$$T_+<\max\{ t\in \RR,\; \Gamma_t\neq \emptyset\}\Longrightarrow u\notin \Lb^5\Lb^{10}(\Gamma_{[t_0,T_+)}).$$
\end{prop}
The maximal time of existence above depends on $\vec{u}_0$, $t_0$ and $\Gamma$ and we will sometimes denote it as $T_+(\vec{u}_0,t_0,\Gamma)$ when there is an ambiguity.

Proposition \ref{Prop:WP} is a consequence of the following small data well-posedness result, which can be proved by Strichartz estimates and a fixed point argument, using finite speed of propagation:
\begin{prop}
\label{P:smalldata}
There exists $\delta_0>0$ with the following property.
 Let $(u_0,u_1)$, $t_0$ and $\Gamma$ be as in Proposition \ref{Prop:WP}. Let $t_1>t_0$. Let $u_L$ be the solution of the free wave equation \eqref{eq:FW} with initial data $(u_0,u_1)$ at $t=t_0$.
 Assume $\|u_L\|_{\Lb^5\Lb^{10}(\Gamma_{[t_0,t_1)})}<\delta_0$. Then there is a solution $u$ of \eqref{eq:NLW} on $\Gamma_{[t_0,t_1)}$. Furthermore
 $$\sup_{t_0\leq t<t_1}\left\|\vec{u}(t)-\vec{u}_L(t)\right\|_{\HHb(\Gamma_{\{t\}})}\lesssim \|(u_0,u_1)\|^5_{\HHb(\Gamma_{\{t_0\}})}.
 $$
\end{prop}
We note that the well-posedness theory sketched above is not time reversible in general. Indeed, to define a backward flow for the equation \eqref{eq:NLW} on $\Gamma$, one must assume that for all $(t_0,x_0)$ in $\Gamma$, the future wave cone
$\{(t,x),\; |x-x_0|\leq t-t_0\}$ is included in $\Gamma$ (in this case $\Gamma$ is called a \emph{domain of dependence}). However, the only nonempty subset of $\RR^4$ which is both a domain of influence and a domain of dependence is $\RR^4$.

We will typically use the theory above to study the wave equations \eqref{eq:LW}, \eqref{eq:NLW} on $\{|x|>|t|+R\}$, with initial data defined for $|x|>R$ at $t=0$, relying on the domain of influence $\{(t,x), \; |x|>t+R\}$ for positive times, and on the domain of dependence $\{(t,x), |x|>-t+R\}$ for negative times.

We conclude this subsection by stating and proving a long-time perturbation theory result on domains of influence.
\begin{thm}
 \label{T:LTPT}
 Let $\Gamma$ be a domain of influence, $-\infty<t_0<t_1\leq \infty$, $I=[t_0,t_1)$. Let $A>0$. There exists $\delta_0=\delta_0(A)$ with the following property. Let $(u_0,u_1),\, (v_0,v_1)\in \HHb(\Gamma_{\{t_0\}})$. Assume that
\begin{gather*}
 \left\|S_L(\cdot-t_0)\big((u_0,u_1)-(v_0,v_1)\big)\right\|_{\Lb^5\Lb^{10}(\Gamma_I)}=\delta\leq \delta_0\\
 \partial_{tt}v-\triangle v=f(v)+e,\quad \vec{v}_{\restriction t=t_0}=(v_0,v_1)\\
 \|e\|_{\Lb^1\Lb^2(\Gamma_I)}=\eps\leq \delta_0,\quad \|v\|_{\Lb^5\Lb^{10}(\Gamma_I)}\leq A.
 \end{gather*}
 Let $u$ be the forward solution of \eqref{eq:NLW} on $\Gamma$, with initial data $(u_0,u_1)$ at $t=t_0$. Then $u$ is defined on $\Gamma_I$ and satisfies:
 \begin{gather}
  \label{u-v1}
  \big\|u-v\big\|_{\Lb^5\Lb^{10}(\Gamma_I)}\lesssim_A \delta+\eps\\
  \label{u-v2}
  \sup_{t\in I}\left\|\vec{u}(t)-\vec{v}(t)\right\|_{\HHb(\Gamma_{\{t\}})}\lesssim_A \left\|(u_0,u_1)-(v_0,v_1)\right\|_{\HHb(\Gamma_{\{t_0\}})}+\delta+\eps.
 \end{gather}
\end{thm}
\begin{proof}
Let $T_+=T_+(\vec{u}_0,t_0,\Gamma)$ be the maximal time of existence of $u$. Let $T_1=\min(t_1,T_+)$. Let $\eta(t)=\|u-v\|_{\Lb^{10}(\Gamma_{\{t\}})}$, so that $\|u-v\|_{\Lb^5\Lb^{10}(\Gamma_I)}=\|\eta\|_{L^5(I)}$ for every interval $I\subset [t_0,T_1)$.

By Strichartz estimates, we have, for $t_0<t<T_1$,
\begin{multline*}
 \|\eta\|_{L^5([t_0,t])}\lesssim \delta
 +\eps+\int_{t_0}^t \left\|f(u(s))-f(v(s))\right\|_{\Lb^2(\Gamma_{\{s\}})}ds\\
 \lesssim \delta+\eps+\int_{t_0}^t \eta(s)\left(\eta(s)^4+\|v(s)\|^4_{\Lb^{10}(\Gamma_{\{s\}})}\right)ds
\end{multline*}
where we have used the Lipschitz property \eqref{ineg_f} of the nonlinearity. We thus have, for a constant $C>0$,
\begin{equation*}
 \|\eta\|_{L^5([t_0,t])}\leq
C\left(\delta+\eps+\|\eta\|_{L^{5}([t_0,t])}^5+\int_0^t \eta(s)\|v(s)\|^4_{\Lb^{10}(\Gamma_{\{s\}})}ds\right).
\end{equation*}
We first assume $C\|\eta\|^4_{L^{5}([t_0,t])}\leq \frac{1}{2}$, so that
\begin{equation*}
 \|\eta\|_{L^5([t_0,t])}\leq
2C\left(\delta+\eps+\int_0^t \eta(s)\|v(s)\|^4_{\Lb^{10}(\Gamma_{\{s\}})}ds\right).
\end{equation*}
Using a Gronwall-type result (see e.g. \cite[Lemma 8.1]{FaXiCa11}) we obtain, for some constant $M(A)$ depending only on $A$,
\begin{equation}
\label{desired_bound}
 \|\eta\|_{L^{5}([t_0,t])}\leq M(A)(\delta+\eps).
\end{equation}
Assuming $\max(\delta,\eps)\leq \delta_0$, where $\delta_0$ is such that $C M(A)^4(2\delta_0)^4=\frac{1}{4}$, we obtain $C\|\eta\|_{L^{5}([t_0,t])}^4\leq \frac{1}{4}$. Thus by a standard bootstrap argument, \eqref{desired_bound} holds for all $t\in [t_0,T_1)$. This shows that $T_1=t_1$ and that \eqref{u-v1} holds. Using \eqref{u-v1}, the energy estimate and H\"older inequality, we deduce \eqref{u-v2}.
\end{proof}
We next give a technical result on the extension of radial solutions defined on exterior of wave cones.
\begin{prop}
\label{P:defR-}
 Let $(u_0,u_1)\in \HHb$, $R\in \RR$
 and
 \begin{equation}
 \label{defGammaR}
\Gamma^R=\{(t,x),\;|x|>t+R\}
 \end{equation}
 Assume that there is a solution $u$ of \eqref{eq:NLW}, with initial data $(u_0,u_1)_{\restriction \Gamma^R_{\{0\}}}$ at $t=0$, defined on $\Gamma^R_{[0,\infty)}$ and such that $u\in \Lb^5\Lb^{10}\left(\Gamma^R_{[0,\infty)}\right)$. Then there is $\rho<R$ such that $u$ can be extended to as solution on $\Gamma^{\rho}_{[0,\infty)}$
 with initial data $(u_0,u_1)_{\restriction \Gamma^{\rho}_{\{0\}}}$ at $t=0$ and such that $u\in \Lb^5\Lb^{10}\left(\Gamma^{\rho}_{[0,\infty)}\right)$.
\end{prop}
\begin{rem}
Define $R_-=R_-(u_0,u_1)$ as the infimum of $R\in \RR$ such that there exists a solution $u$ on $\Gamma^{R_-}$ with initial data $(u_0,u_1)$. Proposition \ref{P:defR-} implies the following analog of the blow-up criterion in the variable $r=|x|$ instead of $t$:
$$ R_->-\infty\Longrightarrow u\notin \Lb^5\Lb^{10}(\Gamma^{R_-}).$$
\end{rem}
\begin{proof}
 We let $\overline{f(u)}$ be the extension of $f(u)$ to $[0,\infty)\times \RR^3$ such that $\overline{f(u)}(t,x)=0$ if $|x|<t+R$. By our assumption $\overline{f(u)}\in \Lb^1\Lb^2([0,\infty)\times \RR^3)$. We let $v$ be the solution of
 $$\partial_{tt}v-\triangle v=\overline{f(u)},\quad \vec{v}_{\restriction t=0}=(u_0,u_1).$$
By Strichartz estimate $v\in \Lb^5\Lb^{10}([0,\infty)\times \RR^3)$ and $\vec{v}\in \mathcal{C}^0([0,\infty),\HHb)$.
 By finite speed of propagation $v=u$ on $\Gamma_{[0,\infty)}^R$. We will solve, on $[0,\infty)\times \RR^3$:
 \begin{equation}
 \label{eq:w_rho}
  \partial_{tt} w-\triangle w=f(w)\indic_{\Gamma^{\rho}},\quad \vec{w}_{t=0}=(u_0,u_1).
 \end{equation}
 By finite speed of propagation, the solution $w$ of \eqref{eq:w_rho} must coincide with $v$ on $\Gamma^{R}_{[0,\infty)}$. Letting $h=w-v$, we see that \eqref{eq:w_rho} is equivalent to:
 \begin{equation}
 \label{eq:h_rho}
  \partial_{tt} h-\triangle h=\left(f(v+h)-f(v)\right)\indic_{\{\Gamma^{\rho}\setminus\Gamma^R\}},\quad \vec{h}_{t=0}=(0,0).
\end{equation}
Letting $\rho<R$ close enough to $R$, so that $f(v)\indic_{\{\Gamma^{\rho}\setminus\Gamma^R\}}$ is small in $\Lb^1\Lb^2$, we see that the equation \eqref{eq:h_rho} can be solved easily by fixed point in a small ball of $\Lb^5\Lb^{10}([0,\infty)\times \RR^3$, which concludes the proof.
\end{proof}

\subsection{Profile decomposition}
\label{sub:profile}
To prove Theorem \ref{T:sequence_resolution} by the channels of energy method, we need to approximate a sequence of solutions of \eqref{eq:NLW} outside wave cones. We do so by using the profile decomposition of Bahouri and G\'erard \cite{BaGe99}.

In \cite{BaGe99}, the profile decomposition is constructed for scalar functions, however it is easy to generalize to vector-valued functions by working on each component. We will only treat the radial case. We thus consider a bounded sequence $\left(u_{0,n},u_{1,n}\right)_n$ in $\HHb$, and let $u_{L,n}$ the corresponding solutions of the linear wave equation \eqref{eq:LW}.
By \cite{BaGe99}, there exists a subsequence of $\left(u_{0,n},u_{1,n}\right)_n$ (that we still denote by $\left(u_{0,n},u_{1,n}\right)_n$) with the following properties.

There exist a sequence $(U^j_{L})_{j\geq 1}$ of radial solutions of the linear equation \eqref{eq:LW} with initial data $(U^j_0,U^j_1)\in \HHb$, and, for $j\geq 1$, sequences $(\lambda_{j,n})_n$, $(t_{j,n})_n$ with $\lambda_{j,n}>0$, $t_{j,n}\in \RR$ satisfying the pseudo-orthogonality relation
\begin{equation}
\label{ortho_param}
j\neq k\Longrightarrow \lim_{n\to \infty} \frac{\lambda_{j,n}}{\lambda_{k,n}}+\frac{\lambda_{k,n}}{\lambda_{j,n}}+\frac{|t_{j,n}-t_{k,n}|}{\lambda_{j,n}}=+\infty.
\end{equation}
such that, denoting
\begin{equation}
\label{simplicity}
U^j_{L,n}(t,x)=\frac{1}{\lambda_{j,n}^{1/2}}U^j_{L}\left(\frac{t-t_{j,n}}{\lambda_{j,n}},\frac{x}{\lambda_{j,n}}\right),
\end{equation}
and
\begin{equation}
\label{decompo_profil}
w_{n}^J(t,x):=u_{n}(t,x)-\sum_{j=1}^J U^j_{L,n}(t,x)
\end{equation}
then
\begin{equation}
\label{small_w}
\lim_{J\rightarrow+\infty}\limsup_{n\rightarrow+\infty} \left\|w_n^J\right\|_{\Lb^{4}\Lb^{12}(\RR^4)}=0.
\end{equation}
One says that $(u_{0,n},u_{1,n})_n$ admits a \emph{profile decomposition} with profiles $\lf\{U_{L,n}^j\rg\}_j$.

 The following expansions hold for all $J\geq 1$:
 \begin{equation}
 \label{pythagore}
 \left\|(u_{0,n},u_{1,n}\rg\|_{\HHb}^2=\sum_{j=1}^J \left\|\left(U^j_{0},U^j_1\right)\rg\|_{\HHb}^2+\left\|(w_{0,n}^J,w_{1,n}^J)\rg\|_{\HHb}^2+o_n(1).
 \end{equation}
The profiles are constructed as follows. Let $v_n(t)=S_L(t)(u_{0,n},u_{1,n})$. Then
\begin{equation}
\label{weak_CV_Uj}
\left(\lambda_{j,n}^{1/2} v_n\left(t_{j,n},\lambda_{j,n}\cdot\right),\lambda_{j,n}^{3/2} \partial_tv_n\left(t_{j,n},\lambda_{j,n}\cdot\right)\right)\xrightharpoonup[n\to \infty]{}(U_0^j,U_1^j).
\end{equation}
weakly in $\HHb$. In other words, the initial data $(U^j_0,U^j_1)$ of the profiles are exactly the weak limits, in $\HHb$, of sequences $$\left(\lambda_n^{1/2}v_n(t_n,\lambda_n \cdot),\lambda_n^{3/2}\partial_tv_n(t_n,\lambda_n \cdot)\right)_n,$$
where $\left(\lambda_n\right)_n$, $\left(t_n\right)_n$ are sequences in $(0,\infty)$ and $\RR$ respectively.
Note that \eqref{weak_CV_Uj} for $j\in \{1,\ldots,J\}$ implies:
\begin{equation}
\label{weak_CV_wJ}
j\leq J\Longrightarrow
\left(\lambda_{j,n}^{1/2} w_n^J\left(t_{j,n},\lambda_{j,n}\cdot\right),\lambda_{j,n}^{3/2} \partial_tw_n^J\left(t_{j,n},\lambda_{j,n}\cdot\right)\right)\xrightharpoonup[n\to \infty]{}0,
\end{equation}
weakly in $\HHb$.

Using this profile decomposition, one can approximate sequences of solutions of the nonlinear wave equation \eqref{eq:NLW} in appropriate subsets of $\RR^4$. We will mainly need this approximation on sets of the form $\{|x|>R+|t|\}$. To state this result, we need to introduce more notations.

We let $(u_{0,n},u_{1,n})_n$ be a bounded sequence in $\HHb$, that admits a profile decomposition as above. Extracting subsequences and translating in time the solutions $U^j$, we can partition the set of indices $\NN^*=\NN\setminus \{0\}$ as $\NN^*=\JJJ^{out}\cup\JJJ^{in}\cup\JJJ^{0}$, where the set of outgoing indices $\JJJ^{out}$, the set of incoming indices $\JJJ^{in}$ and the set of compact indices $\JJJ^c$ are defined by
\begin{align*}
j\in \JJJ^{out}&\iff \lim_{n\to\infty}\frac{-t_{j,n}}{\lambda_{j,n}}=+\infty\\
j\in \JJJ^{in}&\iff \lim_{n\to\infty}\frac{-t_{j,n}}{\lambda_{j,n}}=-\infty\\
j\in \JJJ^0&\iff \forall n,\quad t_{j,n}=0.
\end{align*}
We next fix a sequence $(R_n)_n$ in $[0,\infty)$ and define $\Gamma_n=\{|x|>R_n+|t|\}$. Extracting subsequences again if necessary, we partition $\JJJ^0$ as $\JJJ^0=\JJJ^0_<\cup \JJJ^0_=\cup\JJJ^0_{>}$, where
\begin{gather}
\notag
j\in \JJJ^{0}_{>}\iff \lim_{n\to\infty}\frac{\lambda_{j,n}}{R_n}=+\infty,\quad
j\in \JJJ^{0}_{<}\iff \lim_{n\to\infty}\frac{\lambda_{j,n}}{R_n}=0\\
\label{defR}
j\in \JJJ^0_{=}\iff \lim_{n\to\infty} \frac{\lambda_{j,n}}{R_n}=\frac{1}{R}\in (0,\infty).
\end{gather}
We note that by the orthogonality property \eqref{ortho_param} of the parameters, $\JJJ^0_=$ has at most one element.

For $j\in \JJJ^0$, we consider the solution $U^j$ of \eqref{eq:LW} with initial data $(U_0^j,U^j_1)$ at $t=0$. If $j\in \JJJ^{in}\cup\JJJ^{out}$, we let $U^j=U^{j}_L$ (see Remark \ref{R:profiles} below). We rescale the profiles as
\begin{equation}
\label{def_Ujn}
 U^j_n(t,x)=\frac{1}{\lambda_{j,n}^{1/2}}U^j\left( \frac{t-t_{j,n}}{\lambda_{j,n}},\frac{x}{\lambda_{j,n}} \right).
\end{equation}
\begin{thm}
\label{T:approx}
Let $(u_{0,n},u_{1,n})_n$ and $R_n$ be as above. We assume
\begin{itemize}
\item for all $j\in \JJJ^0_{>}$, there exists $\eps_j$ such that the solution $U^j$ of \eqref{eq:NLW} exists on $\{|x|>|t|-\eps_j\}$, with $U^j\in \Lb^5\Lb^{10}(\{|x|>|t|-\eps_j\})$;
\item for $j\in \JJJ^0_{=}$, there exists $\eps_j>0$ such that $U^j$ exist on $\{|x|>|t|+R-\eps_j\}$, with $U^j\in \Lb^5\Lb^{10}(\{|x|>|t|+R-\eps_j\})$, where $R$ is defined by \eqref{defR}.
\end{itemize}
Then for large $n$, the solution $u_n$ of \eqref{eq:NLW} with initial data $(u_{0,n},u_{1,n})$ at $t=0$ on $\Gamma_n=\{|x|>R_n+|t|\}$ is global in time, satisfies
$$\limsup_{n\to\infty}\|u_n\|_{\Lb^5\Lb^{10}(\Gamma_n)}<\infty$$
and, denoting for all $J$
\begin{equation}
 \label{defrJn}
 \eps_{n}^J(t,x)=u_{n}(t,x)-\sum_{j=1}^JU^j_{n}(t,x)-w_n^J(t,x),
\end{equation}
one has
\begin{equation}
\label{small_epsnJ}
 \lim_{J\to\infty}\limsup_{n\to\infty}\left\|\eps_n^J\right\|_{\Lb^{5}\Lb^{10}(\Gamma_n)}+\sup_{t\in \RR}\left\|\vec{\eps}^J_n\right\|_{\HHb(\Gamma_{n,\{t\}})}=0,
\end{equation}
where $\Gamma_{n,\{t\}}$ is the set $\{x\in \RR^3\text{ s.t. }(t,x)\in \Gamma_n\}$.
\end{thm}
\begin{rem}
 \label{R:profiles}
 For $j\in \JJJ^{out}\cup \JJJ^{in}$, one usually defines $U^j$ as the solution of \eqref{eq:NLW} such that
 \begin{equation}
 \label{defUj}
  \lim_{t\to \pm\infty} \left\|\big(U^j(t),\partial_tU^j(t)\big)-\big(U^j_L(t),\partial_tU^j_L(t)\big)\right\|_{\HHb}=0,
 \end{equation}
 where the sign below the limit is $+$ if $j\in \JJJ^{out}$ and $-$ if $j\in \JJJ^{in}$. In the context of Theorem \ref{T:approx}, there is no arm in defining instead $U^j=U^j_{L}$, since $U^j_L$ and the nonlinear profile defined by \eqref{defUj} are asymptotically close in the energy space, as $n\to\infty$, on $\Gamma_n$. This follows from
 \begin{equation}
\label{small}
\lim_{n\to\infty}
\left\|U^j_{L,n}\right\|_{\Lb^{5}\Lb^{10}(\Gamma_n)}=0,
 \end{equation}
 which is easily checked by a change of variable and the dominated convergence theorem.
\end{rem}
\begin{proof}[Sketch of proof of Theorem \ref{T:approx}]
We let
$$v_n^J=\sum_{j=1}^J U_n^j(t,x)+w_n^J(t,x),$$
and check that for a large fixed $J$, $u_n$, $v_n^J$ satisfy the assumptions of Theorem \ref{T:LTPT} for $n$ large.

The proof is the same as the proof of the main theorem in \cite{BaGe99} (see section 4 of this article), restricting on $\Gamma_n$ by finite speed of propagation and using also the following observations:
\begin{itemize}
 \item One has the inequality
 \begin{multline*}
  \left|f\left(\sum_{j=1}^JU_n^j+w_n^J\right)-\sum_{j=1}^J f(U_n^j)-f(w_n^J)\right|\\
  \lesssim_J \sum_{\substack{1\leq j\leq J\\ j\neq k}} |U_n^j|^4|U_n^k|+|w_n^J|\sum_{j=1}^J |U_n^j|^4+|w_n^J|^4\sum_{j=1}^J |U_n^j|,
 \end{multline*}
as a consequence of the inequality $\left|f(a+b)-f(a)-f(b)\right|\lesssim |a|^4|b|+|b|^4|a|$, which can be proved using that $f$ is $\mathcal{C}^1$ and homogeneous of degree $5$.
\item If $j\in \JJJ^{out}\cup \JJJ^{in}\cup \JJJ^0_{<}$, then $U_n^j$ is defined on $\Gamma_n$ for large $n$  and
$$\lim_{n\to\infty}\left\|U_n^j\right\|_{\Lb^5\Lb^{10}(\Gamma_n)}=0.$$
\item If $j\in \JJJ^0_{=}\cup \JJJ^0_{>}$, $U_n^j$ is defined on $\Gamma_n$ for large $n$ and
$$\limsup_{n\to\infty}\left\|U_n^j\right\|_{\Lb^5\Lb^{10}(\Gamma_n)}<\infty.$$
\end{itemize}
We omit the details.
\end{proof}

We conclude this section by computing the limit of the asymptotic energy
\begin{equation}
\label{def_En}
E_{n}=\lim_{t\to\infty} \int_{|x|>R_n+t}|\nabla_{t,x}u_n(t,x)|^2dx,
\end{equation}
when $n$ goes to infinity, where $u_n$ and $R$ are as in Theorem \ref{T:approx}. We need the following asymptotic formula for the linear wave equation:
\begin{lem}
\label{L:radiative}
For any radial solution $v_L$ of the linear wave equation in dimension $3$ with initial data in $\HHb$, there exists a unique $g\in \Lb^2(\RR)$ such that
\begin{multline}
 \label{radiative}
 \lim_{t\to\infty}\int_{0}^{\infty} \big|\partial_tv_L(t,r)-\frac{1}{r}g(t-r)\big|^2r^2dr\\
 =\lim_{t\to\infty}\int_{0}^{\infty} \big|\partial_rv_L(t,r)+\frac{1}{r}g(t-r)\big|^2r^2dr=0.
\end{multline}
The map $\vec{v}_L(0)\mapsto g$ is an isometry from $\HHb$ to $\Lb^2(\RR)$. We call $g$ the \emph{radiation} associated to $v_L$.
\end{lem}
The proof is given in e.g. \cite[Appendix A]{DuKeMe19}. The fact that $v_L$ is vector valued does not affect the proof, since one can works separately on each component of $v_L$.

For $j\in \JJJ^0_>$, we observe that the assumptions of Theorem \ref{T:approx} imply the existence of a solution $V^j_L$ of the linear wave equation \eqref{eq:LW} such that
\begin{equation}
 \label{defVjL}
 \lim_{t\to\infty}\int_{t}^{\infty} \left|\partial_{t,r} (U^j-V^{j}_L)\right|^2r^2dr=0.
\end{equation}
For $j\in \JJJ^0_=$, there exists a solution $V^j_L$ of \eqref{eq:LW} such that
\begin{equation}
 \label{defVjLbis}
 \lim_{t\to\infty}\int_{R-\eps+t}^{\infty} \left|\partial_{t,r} (U^j-V^{j}_L)\right|^2r^2dr=0,
\end{equation}
where $R$ is defined by \eqref{defR}. For $j\in \JJJ^{out}$, we let $V^j_L=U^j=U^j_L$. In all cases, we define $g^j\in L^{2}(\RR)$ as the radiation associated to $V^j_L$, i.e. such that \eqref{radiative} holds with $v_L=V^j_L$ and $g=g^j$. Letting $G^J_n$ be the radiation associated to $w^J_n$, we see by \eqref{defrJn}, \eqref{small_epsnJ} and the definitions of $V^j_L$, $g^j$ and $G^J_n$, that
\begin{equation*}
E_n=\int_{R_n}^{+\infty} \bigg|\sum_{\substack{1\leq j\leq J\\ j\in \JJJ^0_=\cup \JJJ^0_{>}\cup \JJJ^{out}}} \frac{1}{\lambda_{j,n}^{1/2}}g^{j}\left(\frac{-t_{j,n}-\eta}{\lambda_{j,n}}\right)+G_n^J(-\eta)\bigg|^2d\eta+o^J_n,
\end{equation*}
where
$$\lim_{J\to\infty}\limsup_{n\to\infty}\left|o^J_n\right|=0.$$
We let
\begin{equation}
\label{defrho_j}
 \rho_j=\lim_{n\to\infty}\frac{-t_{j,n}-R_n}{\lambda_{j,n}}.
\end{equation}
Note that $\rho_j=-\infty$ if $j\in \JJJ^0_{<}\cup \JJJ^{in}$, $\rho_j=-R$ if $j\in \JJJ^0_=$ and $\rho_j=0$ if $j\in \JJJ^0_{>}$. If $j\in \JJJ^{out}$, we can always assume that $\rho_j$ exists in $\RR\cup\{\pm\infty\}$ after extraction of a subsequence in $n$. Using the orthogonality of the profiles, and the property \eqref{weak_CV_wJ}, one obtains
\begin{lem}
\label{L:channel}
Let $u_n$ and $R$ be as in Theorem \ref{T:approx}, and $E_n$ be defined by \eqref{def_En}.
With the assumptions and notations above,
\begin{equation}
 \label{limE_n}
 \lim_{n\to\infty} E_n=\sum_{1\leq j\leq J}\int_{-\infty}^{\rho_j}|g^j(\eta)|^2d\eta+\lim_{n\to\infty} \int_{-\infty}^{-R_n}|G_n^J(\eta)|^2d\eta+o(1),\quad J\to\infty.
\end{equation}
\end{lem}

A similar property holds for the past asymptotic exterior energy:
\begin{equation}
E_{n}^-=\lim_{t\to-\infty} \int_{|x|>R_n+|t|}|\nabla_{t,x}u_n(t,x)|^2dx.
\end{equation}
In this case one should exchange the roles of the outgoing profiles ($j\in \JJJ^{out}$) and the incoming ones ($j\in \JJJ^{in}$).

\section{Radial stationary solutions of wave system}
\label{S:stationary}

In this section, we describe some properties of the radial stationary  solutions of \eqref{eq:NLW}, i.e. the radial $\Hb$ solutions of
\begin{equation}
\begin{array}{l}
- \triangle u = f(u),
\end{array}
\label{Eqn:SemilinearEllipt}
\end{equation}
with $u: \mathbb{R}^{3} \rightarrow \mathbb{R}^{m}$, and $f$ satisfying \eqref{H1}. We start by proving that for all $\theta\in \RR^m\setminus \{0\}$, there exists a unique solution of \eqref{Eqn:SemilinearEllipt}, defined for large $r$ such that $\lim_{r\to\infty} rZ_{\theta}(r)=\theta$ (see Subsection \ref{sub:ell_large_r}). In Subsection \ref{sub:rigidity} we obtain as a consequence of a rigidity theorem for equation \eqref{eq:NLW}, proved in the next section, that all $\Hb$, radial solutions of \eqref{Eqn:SemilinearEllipt} are of the form $Z_{\theta}$. As a consequence, we prove that the $\Hb$ solutions of \eqref{Eqn:SemilinearEllipt} on $\RR^3$ are $\mathcal{C}^4$. In subsection \ref{sub:potential}, we consider nonlinearities $f$ that satisfies \eqref{H0}, i.e. that are of the form $\nabla_u F(u)$ for an homogeneous function $F$. We give general properties of the set $\ZZZ$ of nonzero solutions of \eqref{Eqn:SemilinearEllipt}. We then prove that $\ZZZ$ is empty if and only if $F$ is nonnegative, and characterize the ground states elements of $\ZZZ$. Finally, in Subsection \ref{sub:examples} we give some examples.

\subsection{Asymptotic properties of radial, stationary solutions of wave system}
\label{sub:ell_large_r}
In the proposition below we list some properties of the radial solutions of the elliptic system \eqref{Eqn:SemilinearEllipt} that are defined, and in $\Hb$, for large $r$, viewing \eqref{Eqn:SemilinearEllipt} as an equation with an initial data at $r=\infty$. This point of view (initiated in \cite{DuKeMe14}, see Proposition 3.2), is necessary to implement the method of channels of energy in the proof of the rigidity theorem (Theorem \ref{T:rigidity}).
\begin{prop}
\label{P:Ztheta}
Let $\theta \in \mathbb{R}^{m}$. Then the following properties hold:
\begin{enumerate}
\item there exists $R_{\theta} \geq 0 $, $C:= C(\theta) > 0$,  and a $\mathcal{C}^{2}_{rad}-$ function $Z_{\theta}$ of (\ref{Eqn:SemilinearEllipt}) on $|x| > R_{\theta}$ such that
\begin{equation}
\begin{array}{l}
r > R_{\theta}: \;  \left| Z_{\theta}(r) - \frac{\theta}{r} \right| \leq \frac{C(\theta)}{r^{3}}, \; \text{and} \;  \left| Z^{'}_{\theta}(r) + \frac{\theta}{r^{2}} \right| \leq  \frac{C(\theta)}{r^{4}}
\end{array}
\label{Eqn:EstZtheta}
\end{equation}
Moreover we can choose $R_{\theta}$ such that either $(a)$ or $(b)$ or $(c)$ holds with
\begin{itemize}
\item[$(a)$]:  $R_{\theta} > 0$ and $\limsup_{r \rightarrow R_{\theta}^{+}}  \left| Z_{\theta}(r) \right| = \infty $


\item [$(b)$]:  $R_{\theta} = 0$ and $ Z_{\theta} \notin L^{6} \left( \mathbb{R}^{3} \right) $


\item [$(c)$]: $R_{\theta} = 0$, $Z_{\theta} \in \dot{H}^{1}(\mathbb{R}^{3})$, and $Z_{\theta}$ satisfies \eqref{Eqn:SemilinearEllipt} in the sense of distribution on $\RR^3$.
\end{itemize}

\medskip


\item Let $R > 0$ and $u$ be a solution in $\Hb \left( |x| > R \right)$ of (\ref{Eqn:SemilinearEllipt}) such that $\lim \limits_{r \rightarrow \infty }
r u(r) = \theta $ . Then $ u(x) = Z_{\theta}(x) $ for
$|x| > \max \left( R, R_{\theta} \right)$

\end{enumerate}

\label{Prop:CharacSolStatAsymp}
\end{prop}
\begin{proof}
Since $f$ is homogeneous of order $5$, the differential of $f$ is homogeneous of order $4$. Hence the fundamental theorem of calculus applied
to $t \rightarrow f \left( u + t(v-u) \right)$ yields
\begin{equation}
\begin{array}{l}
| f(u) - f(v) | \leq C \left( |u|^{4} + |v|^{4} \right) |u - v| \cdot
\end{array}
\label{Eqn:fufv}
\end{equation}
Let $\theta \in \mathbb{R}^{m}$, $ R \gg 1$, and $X_{R} := \left\{ u \in \mathcal{C}^{0} \left( [R, \infty), \mathbb{R}^{m} \right), \, \| u \|_{R} := \sup \limits_{r \geq R} r |u(r)| \right\}$. \\
Let $\overline{\mathcal{B}} \left( O, 2 |\theta| \right) $ be the closed ball of $X_{R}$ centered at $O$ and with radius $2 |\theta|$. Let
$ u \in \overline{\mathcal{B}} \left( O,2 |\theta| \right) $. We define
\begin{equation}
\begin{array}{l}
\Phi_{\theta}(u)(r):= \frac{\theta}{r}  - \int_{r}^{\infty} s^{-2} \int_{s}^{\infty} \rho^{2} f (u(\rho)) \, d \rho \, ds \cdot
\end{array}
\nonumber
\end{equation}
We claim that $\Phi_{\theta}$ is a contraction. Indeed by direct computation and the claim above we see that for $r \geq R$ and $(u,v) \in X_{R} \times X_{R} $
\begin{equation}
\begin{array}{l}
\left| \Phi_{\theta}(u)(r) - \frac{\theta}{r} \right| \leq \frac{C}{r^{3}} \| u \|_{R}^{5}, \; \text{and} \\
\left| \Phi_{\theta}(u)(r) - \Phi_{\theta}(v)(r) \right| \leq \frac{C}{r^{3}} \| u - v \|_{R} \left( \| u \|^{4}_{R} + \| v \|^{4}_{R} \right)
\end{array}
\nonumber
\end{equation}
So we have a constructed a unique solution $u \in \overline{\mathcal{B}} \left( O, 2 |\theta| \right) $ that satisfies $u(r) = \Phi_{\theta}(u)(r) $ for
$r \geq  R$. Moreover the first estimate of (\ref{Eqn:EstZtheta}) holds. Now by applying twice the fundamental theorem of calculus to $\Phi_{\theta}(u)$ , we see that $u \in \mathcal{C}^{2}([R, \infty))$. Moreover
\begin{equation}
\begin{array}{l}
Z^{'}_{\theta}(r) = - \frac{\theta}{r^{2}} + \frac{1}{r^{2}} \int_{r}^{\infty} \rho^{2} f(u(\rho)) \; d \rho
\end{array}
\end{equation}
Hence the second estimate of (\ref{Eqn:EstZtheta}) holds.

Next we consider a radial solution $u \in \Hb (|x| > R)$ of (\ref{Eqn:SemilinearEllipt}) such that
\begin{equation}
\label{lim_r_ur}
 \lim_{r\to\infty} ru(r)=\theta \in \RR^{m}.
\end{equation}
Clearly, in view of the radial Sobolev embedding
$\Hb (|x| > R) \hookrightarrow \mathcal{C}^{0}(|x| > R)$, (\ref{Eqn:SemilinearEllipt}), elementary properties of Sobolev spaces (such as the product rule and the composition rule), and the formula $ \triangle f (x) = \frac{(r  f)^{''}}{r} $, we see that $u \in \mathcal{C}^{2} (|x| > R)$. Hence integrating once
(\ref{Eqn:SemilinearEllipt}) between $r_{2} \geq r_{1} \gg 1$, using \eqref{lim_r_ur} and $ \triangle f (x) = \frac{(r^{2}  f^{'})^{'}}{r^{2}} $, we see from the Cauchy criterion for integrals that $r^{2} u^{'}(r)$ has a limit as $r \rightarrow \infty$; moreover
it is clear this limit can only be $- \theta$. Hence integrating twice (\ref{Eqn:SemilinearEllipt}) we see that
\begin{equation}
\begin{array}{l}
u(r) = \frac{\theta}{r}  - \int_{r}^{\infty} s^{-2} \int_{s}^{\infty} \rho^{2} f (u(\rho)) \, d \rho \, ds \cdot
\end{array}
\nonumber
\end{equation}
Now from the formula above and $ \lim \limits_{r \rightarrow \infty} ru(r) = \theta $ we see that $u \in \bar{B} \left( O, 2 |\theta| \right)$ for $R \gg 1$. Hence
the application of the fixed point theorem shows that $u=Z_{\theta}$.

We then consider the maximal interval $I_{\theta} := (R_{\theta}, \infty) $ of the solution $Z_{\theta}$ of the ordinary differential equation $z''+\frac{2}{r}z'+f(z)=0$. Thus $Z_{\theta}$ is well defined, and $\mathcal{C}^2$, for $r>R_{\theta}$. In particular it is in $\Hb (|x| >R)$ for all $R>R_{\theta}$. If $R_{\theta} > 0$, the standard ODE blowup criterion ensures that
\begin{equation}
\label{bupcriterion}
 \lim \limits_{r \rightarrow R_{\theta}^{+}} |Z_{\theta}(r)| + |Z_{\theta}^{'}(r)| = +\infty.
\end{equation}
Assume now that $\limsup_{r \rightarrow R_{\theta}^{+}}
|u(r)| < \infty $. Then this means that there exists $M < \infty$ such that $|u(r)| \leq M$. But then from
$ \triangle f (x) = \frac{(r^{2}  \partial_{r}f)^{'}}{r^{2}} $ and (\ref{Eqn:SemilinearEllipt}) we see that for $r,r_{0} > R_{\theta}$,
$r u^{'}(r) - r_{0} u^{'}(r_{0}) = - \int_{r}^{r_{0}} \rho^{2} f(u(\rho)) \; d \rho $. Hence, using also that $f$ is homogeneous of degree $5$ we see from the Cauchy criterion
that the integral (and hence $r u^{'}(r)$) converges as $ r \rightarrow R_{\theta}^{+}$. This clearly contradicts  \eqref{bupcriterion}. Thus in the case where $R_{\theta}>0$, (a) holds.

Assume now that $R_{\theta} =0$. Then if $Z_{\theta} \notin \Lb^{6}(\RR^3)$, we are in case (b). Assuming $Z_{\theta}\in \Lb^6(\RR^3)$, we must prove that we are in case (c). We first prove that $Z_{\theta}$ satisfies the equation \eqref{Eqn:SemilinearEllipt} in the sense of distribution on $\RR^3$.

We will denote $u=Z_{\theta}$ to lighten notations. By standard ODE theory, $u$ is $\mathcal{C}^2$ for $r>R_{\theta}$. It thus satisfies \eqref{Eqn:SemilinearEllipt} in the classical sense on $\RR^3\setminus \{0\}$. We let $\varphi\in C_0^{\infty}(\RR^3)$, and consider $\psi\in C_0^{\infty}(\RR^3)$ such that $\psi(x)=1$ for $|x|\leq 1$ and $\psi(x)=0$ for $|x|>2$. We let $\psi_{\eps}(x)=\psi(x/\eps)$. Then, since $u$ is a solution of \eqref{Eqn:SemilinearEllipt} for $|x|>0$,
$$-\int \triangle u \left( 1-\psi_{\eps}\right)\varphi=\int f(u) \left( 1-\psi_{\eps}\right)\varphi.$$
Integrating by parts we obtain
$$ \int u \varphi\triangle \psi_{\eps} +2\int u\nabla \psi _{\eps}\cdot \nabla \varphi-\int u\left(1-\psi_{\eps}\right)\triangle\varphi=\int f(u) (1-\psi_{\eps})\varphi.$$
One checks easily that
$$\lim_{\eps\to 0}\|\triangle \psi_{\eps}\|_{L^{6/5}}=\lim_{\eps\to 0}\|\nabla \psi_{\eps}\|_{L^{6/5}}=\lim_{\eps\to 0}\|\psi_{\eps}\|_{L^{6}\cap L^{6/5}}=0.$$
Using that $u\in \Lb^6$ (and thus $f(u)\in \Lb^{6/5}$), we deduce
\begin{equation}
- \int_{\mathbb{R}^{3}} Z_{\theta} \triangle \varphi \,dx = - \int_{\mathbb{R}^{3}}  f \left( Z_{\theta} \right) \varphi \,dx,
\nonumber
\end{equation}
which proves that $u=Z_{\theta}$ satisfies \eqref{Eqn:SemilinearEllipt} in the sense of distribution on $\RR^3$.
Hence $\triangle Z_{\theta} \in L^{\frac{6}{5}}$. Now by the interpolation inequality
$ \| \nabla Z_{\theta} \|_{L^{2}} \lesssim \| \triangle Z_{\theta} \|^{\frac{1}{2}}_{L^{\frac{6}{5}}} \| Z_{\theta} \|^{\frac{1}{2}}_{L^{6}} $ we conclude that $Z_{\theta}\in \Hb$.
\end{proof}
\begin{exemple}
 \label{Ex:W}
 We recall that $W$ (defined in \eqref{def_W}) is solution to the equation $-\triangle W=W^5$ on $\RR^3$, and that it is (up to scaling and multiplication by $\pm 1$) the only radial solution of this equation. Also, $W$ is the maximizer for the critical Sobolev inequality
 \begin{equation}
  \label{CriticalSobolev}
  \forall \varphi\in \dot{H}^1(\RR^3), \quad \left(\int |W|^6\right)^2\int |\varphi|^6\leq \left(\int |\nabla \varphi|^2\right)^3,
 \end{equation}
 (see \cite{Aubin76, Talenti76})  with equality if and only if $\varphi(x)=\lambda W(\mu x)$, where $\lambda\in \RR$, $\mu>0$. It is natural to look for solutions of \eqref{Eqn:SemilinearEllipt} of the form $\zeta W_{(\lambda)}$, $\zeta\in \RR^m$. It is easy to check that $\zeta W_{(\lambda)}$ is a solution if and only if $f(\zeta)=\zeta$. In this case, by the uniqueness in Proposition \ref{P:Ztheta} and the asymptotics of $W$, we have $\zeta W_{(\lambda)}=Z_{\theta}$, where $\theta= \sqrt{3\lambda}\zeta$. Let us mention that this type of solutions always exist when $f$ satisfies \eqref{H0} and $\ZZZ$ is not empty, see Proposition \ref{P:GroundState}.
 In the case of the focusing Euclidean nonlinearity
$$ f(u)=|u|^4 u,\quad |u|^2=\sum_{j=1}^m u_j^2,$$
all $W_{(\lambda)} \omega $, $\omega \in \Sb^{m-1}$, $\lambda>0$, are elements of $\ZZZ$. We will see in the next section that these are the only elements of $\ZZZ$ in this case (see Exemple \ref{Ex:Euclidean}).
\end{exemple}

\subsection{Rigidity theorem and consequences}
\label{sub:rigidity}
In the next section we will prove the following rigidity theorem, which is one of the crucial step for the proof of Theorems \ref{T:sequence_resolution} and Theorems \ref{T:continuous_resolution}.
\begin{thm}

\label{T:rigidity}

Let $R \geq 0$ and $u$ be a radial solution of $ \partial_{tt} u - \triangle u = f(u) $  on $\{|x|>R+|t|\}$ with radial initial data $(u_{0},u_{1}) \in \HHb( |x|> R)$. Assume that
\begin{equation}
\label{nonradiative}
 \sum \limits_{ \iota \in \{ + , - \} } \lim \limits_{ t \rightarrow \iota \infty } \int_{|x| \geq R + |t|}
 | \nabla_{x,t} u (t,x)|^{2} \, dx  = 0  \cdot
\end{equation}
Then there exists $\theta \in \mathbb{R}^{m}$ such that
\begin{equation}
\begin{array}{l}
\left( u(t,x), \partial_{t} u (t,x) \right) :=  \left( Z_{\theta}(|x|),0 \right), \quad |x|>R.
\end{array}
\nonumber
\end{equation}
\end{thm}

A solution of \eqref{eq:NLW} such that \eqref{nonradiative} holds is called \emph{nonradiative}. We refer to \cite{CoDuKeMe22Pb} for a study of radial nonradiative solutions in general dimensions, in a scalar context.

We give here the consequences of Theorem \ref{T:rigidity} on the set of stationary solutions $\ZZZ$.

\begin{corol}
\label{Cor:UniquenessElliptic}
Let $R > 0 $ and $u$ be a solution of (\ref{Eqn:SemilinearEllipt}) lying in $ \mathbf{\dot{H}}_{rad}(|x| > R)$. Then there exists a (unique) $ \theta \in \mathbb{R}^{m} $ such that
\begin{equation}
\begin{array}{l}
|x| > \max(R,R_{\theta}): \; u(x) =  Z_{\theta}(x) \cdot
\end{array}
\label{Eqn:uZtheta}
\end{equation}
\end{corol}
\begin{proof}
If $R > 0$ then let $\Psi_{R}$ be defined as follows: $\Psi_{R}(f,g)(r) := (f(r),g(r))$ if $r \geq R$ and $\Psi_{R}(f,g)(r) := (f(R),0)$ if $r < R$. If $R=0$ then
let $\Psi_{0}(f,g) = (f,g)$. Observe that by finite speed of propagation, $u$ with data $\Psi_{R}(u,0)$ satisfies the assumptions of Theorem \ref{T:rigidity}. Hence by application of this theorem we get (\ref{Eqn:uZtheta}) for some  $\theta \in \mathbb{R}^{m}$. The uniqueness of $\theta$ follows from the uniqueness of limits.
\end{proof}
\begin{exemple}
\label{Ex:Euclidean}
 Let $m\geq 2$, $f(u)=|u|^4u$ as in  Example \ref{Ex:W}. We have seen that for all $\omega\in \Sb^{m-1}$, $\lambda>0$, $\omega W_{(\lambda)}\in \ZZZ$. Since in this case $\omega W_{(\lambda)}=Z_{\theta}$, where $\theta=\sqrt{2\lambda}\,\omega$, we deduce from Corollary \ref{Cor:UniquenessElliptic}
 $$ \ZZZ=\left\{\omega W_{(\lambda)},\; \lambda>0,\; \omega\in \Sb^{m-1}\right\}.$$
\end{exemple}

From Corollary \ref{Cor:UniquenessElliptic} and Proposition \ref{P:Ztheta} we infer the following corollary:
\begin{corol}
The set $\ZZZ$ is exactly the set of $Z_{\theta}$, $\theta \in \RR^m\setminus\{0\}$ such that $Z_{\theta}\in \mathbf{\dot{H}}_{rad}(\RR^3)$ i.e. such that $Z_{\theta}$ is in case (c) of Proposition \ref{P:Ztheta}.
\label{Corol:ZAsymp}
\end{corol}
Thus any element $Z$ of $\ZZZ$ satisfies
$$ \lim_{r\to\infty} r Z(r)=\lim_{r\to\infty}-r^2 Z'(r)=\theta$$
for some $\theta$ in $\RR^m\setminus \{0\}$. Moreover, for every $\theta\in \RR^m\setminus \{0\}$, there is at most one $Z$ such that the preceding limits hold.

We conclude by proving that the elements of $\ZZZ$ are regular:
\begin{corol}
 \label{C:regular}
 Let $u\in \ZZZ$. Then the Kelvin transform of $u$,
 $$ v:r\mapsto \frac{1}{r}u\left(\frac{1}{r}\right)$$
 is an element of $\ZZZ$. Furthermore $u\in \mathcal{C}^4(\RR^3)$.
\end{corol}
\begin{proof}
To prove that $u\in \mathcal{C}^{4}(\RR^3,\RR^m)$, it is sufficient to prove that $u\in \mathcal{C}^0(\RR^3)$, the conclusion will then follow by standard elliptic theory and a bootstrap argument, using that $f$ is $\mathcal{C}^2$. Since there exists $\theta \in \RR^m\setminus \{0\}$ such that $u=Z_{\theta}$, we already know (see Proposition \ref{P:Ztheta}) that $u\in \Hb(\RR^3)$, and thus that it is continuous on $\RR^3\setminus \{0\}$. Using the equation $\frac{d^2}{dr^2}(r^2u)=r^2f(u)$, we obtain that $u$ is $\mathcal{C}^4$ on $\RR^3\setminus \{0\}$.

The Kelvin transform $v$ of $u$ is in $\mathcal{C}^4(\{|x|>R\})$ for all $R>0$. By straightforward computations, using the properties of $u$, we see that $v$ is in $\Lb^6(\RR^3)$, in $\Hb(\{|x|>R\})$ for all $R>0$, and that it satisfies the equation $-\triangle v=f(v)$ on $\RR^3\setminus\{0\}$. By Corollary \ref{Eqn:uZtheta}, $v=Z_{\tilde{\theta}}$ for some $\tilde{\theta}\in \RR^m\setminus \{0\}$. Since $v\in \Lb^6(\RR^3)$, we are in case $(c)$ of Proposition \ref{P:Ztheta}, which shows that $v\in \ZZZ$. Also, we observe that
$$ \lim_{r\to 0} u(r)=\lim_{r\to 0} \frac{1}{r}v\left( \frac{1}{r} \right)=\tilde{\theta},$$
and thus $u$ can be extended to a continuous function on all $\RR^3$. This concludes the proof.
\end{proof}

\subsection{Properties of the set of stationary states for a potential-like nonlinearity}
\label{sub:potential}

In this subsection we list and prove some properties of $\mathcal{Z}$ when the nonlinearity is of the form $f=\nabla_{\RR^m} F$. We start by a general description of $\ZZZ$ (Proposition \ref{Prop:PropZ}). We will then prove that $\ZZZ$ is empty if and only $F$ is nonpositive, and give an explicit element of $\ZZZ$ when it is not empty (see Proposition \ref{P:GroundState} below).

\begin{prop}
Assume that $f$ satisfies \eqref{H0} and let $F$ such that $f=\nabla_{\RR^m}F$. Then
\begin{enumerate}
\item \label{I:Z1}for all $ Q \in \mathcal{Z}$, $ \int |\nabla Q|^2 \,dx =  6 \int F(Q) \,dx $ and $E(Q) = \frac{1}{3} \int_{\mathbb{R}^{3}} | \nabla Q(x) |^{2} \,dx$.
\item \label{I:Z2} There exists $m>0$ such that for all $Q\in \ZZZ$, $E(Q)\geq m$.
\item \label{I:Z3} $\mathcal{Z}$ is closed in $\Hb$. $\mathcal{Z}\cup \{0\}$ is closed in $\Hb$-weak.\end{enumerate}
\label{Prop:PropZ}
\end{prop}

\begin{proof}
Multiplying  $ - \triangle Q = f(Q) $ by $Q$, integrating by parts, and using the
equality $ u \cdot f(u) =  6  F(u)$ (this is derived from the differentiation with respect to $\lambda$ of
$ F(\lambda u) = \lambda^{6} F(u) $ ) we get
\begin{equation}
\begin{array}{l}
\int_{\mathbb{R}^{3}} |\nabla Q|^{2} \,dx  = \int_{\mathbb{R}^{3}} Q f(Q) = 6 \int_{\mathbb{R}^{3}} F(Q) \,dx  \cdot
\end{array}
\label{Eqn:PohozaevEq}
\end{equation}
Hence $E(Q) = \frac{1}{3} \int_{\mathbb{R}^{3}} |\nabla Q |^{2} \,dx $. \\
\\
We then prove \eqref{I:Z2}. Indeed from the Sobolev embedding $\dot{H}^{1} \hookrightarrow L^{6} $,
the bound $|F(u)|\lesssim |u|^6$, and (\ref{Eqn:PohozaevEq}) we see that there exists $C_{*} > 0 $ such that
\begin{equation}
\begin{array}{l}
\int_{\mathbb{R}^{3}} |\nabla Q|^{2} \,dx \leq C_{*}  \left( \int_{\mathbb{R}^{3}} |\nabla Q|^{2} \;  dx \right)^{3} \cdot
\end{array}
\nonumber
\end{equation}
Hence $Q = 0$ or $\int_{\mathbb{R}^{3}} |\nabla Q |^{2} \,dx \gtrsim 1$. Since by the definition of $\ZZZ$, $O\notin \ZZZ$, we deduce that the second case holds. This yields \eqref{I:Z2} using \eqref{I:Z1}.

We then prove that $\mathcal{Z}$ is closed in $\Hb$. Indeed if $Q_{k} \in \mathcal{Z}$ and $Q_{k} \rightarrow Q $
in $\Hb$ as $k \rightarrow \infty$, we have $E(Q)\geq m$, and thus $Q$ is not identically $0$. Furthermore
$\triangle Q_{k} \rightarrow \triangle Q $ in $ \dot{H}^{-1} \hookrightarrow
L^{\frac{6}{5}} $. On the other hand $f(Q_{k}) \rightarrow f(Q)$ in $L^{\frac{6}{5}}$: this follows from (\ref{Eqn:fufv}) and
Hence $-\triangle Q=f(Q)$, which shows $Q \in \mathcal{Z}$.

We next prove that $\mathcal{Z}\cup \{0\}$ is closed in $\Hb-$ weak. Indeed assume that $Q_{k} \rightharpoonup Q$. Observe that The Rellich-Kondrachov compactness theorem shows that $Q_{k} \rightarrow Q$ in $L^{5}_{loc}$. Thus  $f(Q_{k}) \rightarrow f(Q)$ in $L^{1} \subset L^{1}_{loc}$: this follows from (\ref{Eqn:fufv}).
Letting $k \rightarrow \infty $ in  $ \int_{\mathbb{R}^{3}} \nabla Q_{k} \cdot \nabla \phi \,dx = \int_{\mathbb{R}^{3}} f(Q_{k}) \phi \,dx $ for
$ \phi \in \mathcal{C}_{c}^{\infty} (\mathbb{R}^{3})$ test function, we get $ - \triangle Q  = f(Q) $. Since $Q\in \Hb$, we deduce that $Q\in \ZZZ\cup \{0\}$.


\end{proof}
%
%

\begin{prop}
 \label{P:GroundState}
 Assume that $f$ satisfies \eqref{H0}. Then
 \begin{equation}
 \label{CNS_emptyZ}
  \ZZZ\neq \emptyset \iff \exists u\in \RR^m \text{ s.t. } F(u)>0.
 \end{equation}
 Furthermore, if \eqref{CNS_emptyZ} is satisfied, and $\omega\in \Sb^{m-1}$ is such that $$F(\omega)=\max_{u\in \Sb^{m-1}} F(u),$$ then letting
 \begin{equation}
  \label{CNS10}
  \mu=\left( 6F(\omega) \right)^{-1/4}, \quad Z(x)=\mu W(x)\omega,
 \end{equation}
 we have that $Z\in \ZZZ$ and
 \begin{multline}
  \label{CNS11}
  \int F(Z)=\min_{Q\in \NNN} \int F(Q),\\
  \text{ where } \NNN=\left\{ Q\in \Hb\setminus \{0\},\; 6\int F(Q)=\int |\nabla Q|^2\right\}.
 \end{multline}
 Conversely any solution $Z\in \Hb\setminus \{0\}$ of the minimization problem \eqref{CNS11} is of the form \eqref{CNS10} where $\omega$ is a maximum of $F$ on $\Sb^{m-1}$.
\end{prop}
\begin{rem}
 Observe that $F$ always have a maximum on $\Sb^{m-1}$ by compactness. Moreover, if there exists $u$ such that $F(u)>0$, then by the homogeneity of $F$, the corresponding maximal value of $F$ on $\Sb^{m-1}$ is positive. Thus the proposition shows that if \eqref{CNS_emptyZ} is satisfied, $\ZZZ$ always has an element of the form 
 $z W$, $z\in \RR^{m}\setminus\{0\}$.
\end{rem}
\begin{rem}
Since by Proposition \ref{Prop:PropZ}, $E(Q)=2\int F(Q)$ for $Q\in \ZZZ$, we see that the solutions $Z$ of the minimizing problem \eqref{CNS11} are ground states stationary solutions of \eqref{eq:NLW}, in the sense that they are nonzero stationary solutions of minimal energy. The proposition shows that these ground states always exist when $\ZZZ$ is not empty, and that they are of the form \eqref{CNS10}, where $\omega\in \Sb^{m-1}$ is any maximizer of $F$.
\end{rem}
\begin{exemple}
If $m\geq 2$ and $F$ is given by $F(u)=\frac{1}{6}\sum_{j=1}^m u_j^6$ (i.e. \eqref{eq:NLW} is a system of $m$ decoupled focusing equation), then for any $(\lambda_j)_j \in \RR^m\setminus \{0\}$, $\left(W_{(\lambda_j)}\right)_{1\leq j\leq m}$ is an element of $\ZZZ$. This shows that in this case, there are elements of $\ZZZ$ that are not of the form $\omega W$, $\omega\in \RR^m$. However, we do not know any example of nonlinearity $f$ that satisfies \eqref{H0} such that $\ZZZ$ contains elements that are not of the form
$$\sum_{j=1}^J \omega_j W_{(\lambda_j)}, \quad \omega_j\in S^m, \; \lambda_j>0.$$
 Note that when $f(u)=(u_1^5,5u_1^4u_2)$ (that satisfies \eqref{H1} but not \eqref{H0}), $(W,x\cdot \nabla W+\frac{1}{2}W)$ is an element of $\ZZZ$ which is not of the preceding form. See also Example \ref{Ex:notH0} below where we describe $\ZZZ$ in this case, and also the other examples in Subsection \ref{sub:examples}.
 \end{exemple}

\begin{proof}[Proof of Proposition \ref{P:GroundState}]
\emph{Step 1.}  We first assume that for all $u\in\RR^m$, $F(u)\leq 0$. Let $Q\in \ZZZ$. Then by Proposition \ref{Prop:PropZ}, \eqref{I:Z1}, $\int |\nabla Q|^2\leq 0$, which shows that $Q$ is identically $0$, contradicting the definition of $\ZZZ$.

\smallskip

In the next $3$ steps, we assume
\begin{equation}
 \label{F(u)>0}
 \exists u\in \RR^m,\quad F(u)>0.
\end{equation}

\noindent \emph{Step 2.} By the homogeneity of $F$, there exists $u\in \Sb^{m-1}$ such that $F(u)>0$. Let $\omega\in \Sb^{m-1}$ such that
\begin{equation}
\label{defomega}
F(\omega)=\max_{u\in \Sb^{m-1}} F(u).
\end{equation}
 In this step we show that $Z$ defined by \eqref{CNS10} is an element of $\ZZZ$. Since $\omega$ is a maximum of $F$ on $\Sb^{m-1}$, the gradient of $F$ at $\omega$ must be normal to the tangent space of $\Sb^{m-1}$ at $\omega$. In other terms, there exist $a \in \RR$ such that $f(\omega)=a\omega$. Using that 
 $\omega \cdot f(\omega)=6F(\omega)$ (this comes from differentiating $ F(\lambda \omega) = \lambda^{6} F(\omega) $ at $\lambda =1$), we obtain $a=6F(\omega)>0$.  As a consequence $f(\omega/a^{1/4})=\omega/a^{1/4}$, which shows (see Example \ref{Ex:W}) that $Z=\mu \omega W$, $\mu=a^{-1/4}$, is an element of $\ZZZ$. By Proposition \ref{Prop:PropZ},  $6\int F(Z)=\int |\nabla Z|^2$. Thus $6\mu^6 F(\omega)\int W^6=\mu^2\int |\nabla W|^2$, which yields \eqref{CNS10} since $\int |\nabla W|^2=\int W^6$.

 \smallskip

 \noindent\emph{Step 3.} We next prove that the solution $Z$ defined in the preceding step satisfies \eqref{CNS11}. We note that
 \begin{equation}
  \label{CNS29}
  \int F(Z)=\mu^6 F(\omega)\int W^6=6^{-3/2}F(\omega)^{-1/2}\int W^6.
 \end{equation}
 Let $Q\in \Hb\setminus \{0\}$ such that $6\int F(Q)=\int |\nabla Q|^2$. We have
 \begin{multline}
  \label{CNS30}
  \int F(Q)=\int_{Q\neq 0} |Q|^6 F\left( \frac{Q}{|Q|} \right)\leq F(\omega)\int |Q|^6\\
  \leq F(\omega) \left( \int \left|\nabla |Q|\right|^2 \right)^3\left( \int W^6\right)^{-2},
 \end{multline}
where we have used the critical Sobolev inequality \eqref{CriticalSobolev} on $\varphi=|Q|$. By the inequality $\nabla |Q|\leq |\nabla Q|$ a.e. (see \cite{LiebLoss01BO}), then the assumption $\int |\nabla Q|^2=6\int F(Q)$, we deduce
\begin{equation}
 \label{CNS31}
 \int F(Q)\leq F(\omega) \frac{\left( \int |\nabla Q|^2 \right)^3}{\left( \int W^6 \right)^{2}}=F(\omega) \frac{\left( 6\int F(Q) \right)^3}{\left( \int W^6 \right)^{2}}
\end{equation}
Hence
$$ \left( \int W^6\right)^2\leq 6^3 F(\omega) \left( \int F(Q) \right)^2, $$
which yields, in view of \eqref{CNS29}, the desired inequality $\int F(Z)\leq \int F(Q)$, which proves \eqref{CNS11}.
\smallskip

\noindent\emph{Step 4.} To conclude the proof of the proposition, we must show that if $Q\in \Hb\setminus\{0\}$ is such that $6\int F(Q)=\int |\nabla Q|^2$ and
\begin{equation}
\label{minimization}
\int F(Q)= \min_{R\in \NNN} \int F(R),
\end{equation}
where $\NNN$ is defined in \eqref{CNS11}, then (up to a rescaling) $Q=\mu \omega W$, where  $\omega$, $\mu$ are in \eqref{CNS10}, \eqref{defomega}. Indeed, in this case, all the inequalities in the preceding step must be equalities. Assuming equality in the second inequality in \eqref{CNS30}, we obtain $|Q|=\mu W_{(\lambda)}$ for some $c>0$, $\mu>0$ (see Example \ref{Ex:W}). Since we must have equality (almost everywhere) in the inequality $\nabla |Q|\leq |\nabla Q|$, we obtain (see \cite{LiebLoss01BO}) that $Q=|Q| \omega$ for a fixed $\omega \in \Sb^{m-1}$. Thus $Q=\mu \omega W_{(\lambda)}$, and the fact that $Q$ satisfies the minimization problem \eqref{minimization} shows that $\mu$ is as in \eqref{CNS10}, concluding the proof.
\end{proof}

We next give additional necessary conditions for a function $Q$ to be in $\ZZZ$:
\begin{prop}
Assume \eqref{H0}. Let $Q\in \mathcal{Z}$. Then
$$ F(Q(0))>0 \text{ and }F(\theta)>0,$$
where $\theta\in \RR^m$ is such that $\lim_{r\to\infty} rQ(r)=\theta$.
\end{prop}
\begin{rem}
Recall that Corollary \ref{Cor:UniquenessElliptic} shows that there exists a (unique)  $\theta \in \mathbb{R}^{m}$ such that
$\lim_{r\to\infty} rQ(r)=\theta$.
\label{Rem:ThetaStatSol}
\end{rem}
\begin{proof}
Let  $Q \in \mathcal{Z}$. Then we see from Corollary \ref{C:regular}, that $Q \in \mathcal{C}^{4}(\RR^3)$. Hence
\begin{equation}
\begin{array}{l}
\frac{d^{2}}{d r^{2}}Q+\frac{2}{r}\frac{d}{dr} Q + r f(Q) = 0, r \geq 0
\end{array}
\label{Eqn:EQZCond}
\end{equation}
$ Q'(0) = 0 $, $Q\in \mathcal{C}^{2}([0,\infty),\RR^m$. We recall also that $Q=Z_{\theta}$ and thus, from proposition \ref{P:Ztheta}
$$ \lim_{r\to\infty}r^2\frac{dQ}{dr}=-\theta$$.
Multiplying \eqref{Eqn:EQZCond} by $\frac{dQ}{dr}$, we obtain,
\begin{equation}
\begin{array}{l}
\frac{d}{dr} \left( \frac{1}{2} Q_{r}^{2} + F(Q) \right) = - \frac{2}{r}(Q_{r})^{2}
\end{array}
\nonumber
\end{equation}
Hence, integrating the above equality from $0$ to $r$, and letting $ r \rightarrow \infty $, we see that $F(Q(0)) > 0$.

Let $V$ be the image of $Q$ by the Kelvin transform, i.e $ V(r) := \frac{1}{r} Q \left( \frac{1}{r} \right)$. From Corollary \ref{C:regular}, $V\in \ZZZ$. Furthermore, $ V(0) = \theta $. Hence we see from the previous
result that $ F(V(0))=F(\theta) > 0 $.
\end{proof}

\subsection{More examples}
\label{sub:examples}
In this Subsection we give other examples of nonlinearity satisfying \eqref{H1} or the stronger \eqref{H0} where we can compute $\ZZZ$, or at least give some elements of $\ZZZ$.
\begin{exemple}
\label{Ex:notH0}
Assume $m=2$ and $f(u) :=  \left( u_{1}^{5},5 u_{1}^{4} u_{2}\right) $, which satisfies \eqref{H1} but not \eqref{H0}. Let $(u_1,u_2)$ be an element of $\ZZZ$. Then $u_1$ satisfies
$-\triangle u_1=u_1^5$. Thus $u_1$ is identically $0$ or $u_1$ is of the form $\pm W_{(\lambda)}$ for some sign $\pm$ and some scaling factor $\lambda>0$. In the first case, we obtain that $-\triangle u_2=0$ with $u_2\in \dot{H}^1(\RR^3)$. Thus $(u_1,u_2)=(0,0)$, contradicting the definition of $\ZZZ$. In the second case, we see that $(\pm W_{(\lambda)},u_2)$ is in $\ZZZ$ if and only if $u_2$ is solution of the equation
$$ -\triangle u_2=5 W_{(\lambda)}^4u_2.$$
Note that $u_2=a Y_{(\lambda)}$, $a\in \RR$, give a family of solution of this equation, with $Y=x\cdot W+\frac{1}{2}W$. As a consequence, we have
\begin{equation}
\label{inclusion}
\left\{ \left(\pm W_{(\lambda)}, a Y_{(\lambda)}\right),\; \lambda>0,\; a\in \RR\right\} \subset \ZZZ.
\end{equation}
We claim that
\begin{equation}
\label{equality}
\left\{ \left(\pm W_{(\lambda)}, a Y_{(\lambda)}\right),\; \lambda>0,\; a\in \RR\right\} =\ZZZ.
\end{equation}
Indeed, looking at the asymptotics of $W_{(\lambda)}$ and $Y_{(\lambda)}$, we see that the left-hand side of \eqref{inclusion} is
$$ \left\{Z_{\theta},\; \theta\in \RR^2\setminus \{(0,\sigma),\; \sigma \in \RR\} \right\}.$$
Furthermore, if $\theta=(0,\sigma)$ with $\sigma\in \RR$, we see that $Z_{\theta}(r)=\left(0,\frac{\sigma}{r}\right)\notin \Lb^6(\RR^3)$, and thus that $Z_{\theta}\notin \ZZZ$. This concludes the proof of \eqref{equality}.
\end{exemple}

\begin{exemple}
 \label{Ex:mixte}
 If $m=2$ and $F(u)=\frac{1}{3}u_1^3u_2^3$, that is $f(u)=\left(
                 u_1^2u_2^3,u_2^2u_1^3\right),$
it is easy to check that the maximum value  of $F$ on $\Sb^{1}$ is $\frac{1}{24}$, and that it is attained at the points $\pm \left(\frac{1}{\sqrt{2}},\frac{1}{\sqrt{2}}\right)$. This corresponds to the two families of ground state solutions $\pm (W_{(\lambda)},W_{(\lambda)})$, $\lambda>0$. As a consequence of Proposition \ref{Prop:DefocDirecProp} below, $\ZZZ$ is exactly $\left\{ \pm\left(W_{(\lambda)},W_{(\lambda)} \right),\; \lambda>0\right\}$.
\end{exemple}
\begin{prop}
Let $f(u)$ satisfy \eqref{H1}. Let $\omega\in \RR^m\setminus \{0\}$
such that
for all $ u \in \mathbb{R}^{m}$
\begin{equation}
\left( f(u) \cdot \omega \right) \left( u \cdot \omega  \right) \leq 0.
\label{Eqn:CondDefoc}
\end{equation}
(In this case we will say that $\omega$ is a \emph{defocusing direction} for $f(u)$). Then if $\theta$ is such that $Z_{\theta}\in \ZZZ$, we have
$\omega \cdot \theta = 0$.
\label{Prop:DefocDirecProp}
\end{prop}
\begin{exemple}
Consider again the nonlinearity of Example \ref{Ex:mixte}
Then the direction $(1,-1)$ is clearly defocusing. 
As a consequence of Proposition \ref{Prop:DefocDirecProp}, if $Z_{\theta}\in \ZZZ$, then $\theta$ is orthogonal to $(1,-1)$. Thus
$\theta=(\sigma,\sigma)$, where $\sigma\in \RR\setminus \{0\}$. Note that in this case
$Z_{\theta}=\pm (W_{(\lambda)},W_{(\lambda)})$, where $\pm$ is a sign of $\sigma$ and $\sqrt{\lambda 3}=|\sigma|$.

Since we have already proved that $\{\pm (W_{(\lambda)},W_{(\lambda)}),\; \lambda>0\}\subset \ZZZ$, we deduce that $\ZZZ$ is exactly $\{\pm (W_{(\lambda)},W_{(\lambda)}),\; \lambda>0\}$.
\end{exemple}

\begin{proof}[Proof of Proposition \ref{Prop:DefocDirecProp}]
Assume that $u=Z_{\theta}\in \ZZZ$.
Let $v:= u \cdot \omega$. Then from (\ref{Eqn:SemilinearEllipt}) we get
\begin{equation}
- v^{''}  - \frac{2}{r} v^{'} = f(u) \cdot \omega
\label{Eqn:ODEv}
\end{equation}
We also have $v(r) \approx \frac{\theta \cdot \omega}{r}$ and $v^{'}(r) \approx - \frac{\theta \cdot \omega}{r^{2}}$. We may assume WLOG that
$\theta \cdot \omega > 0$. Consequently $v(r) > 0$ and $v^{'}(r) < 0$ for $r \gg 1$. We next prove that these inequalities hold for all $r > 0$. Assume toward a contradiction
that this statement does not hold. Then this means that there exists $r_{0} > 0 $ such that either $(i)$ or $(ii)$ holds with
$(i)$: $v(r_{0}) = 0 $ and $v^{'}(r_{0}) < 0$ ; $(ii)$: $v(r_{0}) > 0 $ and $v^{'}(r_{0}) = 0$; furthermore $v(r) > 0$  and $v^{'}(r) < 0 $
for $r > r_{0}$. Clearly $(i)$ cannot occur. In view of (\ref{Eqn:CondDefoc}) and (\ref{Eqn:ODEv}) we see that $v^{''}(r) \geq 0$ for $r \geq r_{0}$: therefore
$(ii)$ cannot either occur. \\
This shows that $v(r) > 0$ and that $v^{'}(r) < 0 $ for $r > 0$; moreover $v^{''}(r) > 0 $ for $r > 0$. Hence there exist $l \in [-\infty, 0) $ such that $v^{'}(r) \rightarrow l $ as $r \rightarrow 0$. From (\ref{Eqn:CondDefoc}) and (\ref{Eqn:ODEv}) we get $v^{''}(r) \geq -\frac{2}{r} v^{'}(r)$; hence $ \int_{0}^{1} v^{''}(r) dr = + \infty $  and
$l =  - \infty$. Hence $ - \frac{v^{''}}{v^{'}} \geq \frac{2}{r}$,
$ \frac{d}{dr} \left( \log(r^{2}) + \log  (|v^{'}|(r)) \right)  \leq 0 $, and by integration $r^{2}|v^{'}(r)| \geq |v^{'}(1)|$ for $r \leq 1$. Hence $u \notin
\Hb $, which is a contradiction.
\end{proof}

\begin{exemple}
 If $m=2$ and $F(u)=\frac 16|u|^5u_1$, $f(u)=\frac{1}{6}|u|^5 \left(1 , 0\right) + \frac 56|u|^3u_1u$, $F$ has only one maximum point on $\Sb^{1}$, at the point $\left(1,0\right)$. This corresponds to the family of ground state solutions $\left( W_{(\lambda)}, 0\right)$, $\lambda>0$. Since the only solution of the equation $f(u)=u$ is $\left(1,0\right)$,
 these are the only elements of $\ZZZ$ of the form $\omega W_{(\lambda)}$, $\omega\in \RR^m\setminus \{0\}$. We do not know, in this case, if there exist other elements of $\ZZZ$.
\end{exemple}

\section{Proof of the rigidity theorem}
\label{S:rigidity}
In this section we prove Theorem \ref{T:rigidity}.

We only sketch the proof, which is very close to the proof of analogous results in the scalar case. See e.g. \cite[Subsection 2.2]{DuKeMe13}, \cite[Subsection 3.2]{DuKeMe14}. The proof uses the following property of the wave equation:
\begin{prop}
\label{P:channels}
Let $R\geq 0$.
 Let $(u_0,u_1)\in \HHb$ and $u_L(t)$ the corresponding solution of \eqref{eq:LW}. Then
 \begin{multline}
 \label{lower_bound}
  \sum_{\pm}\lim_{t\to\pm \infty} \int_{R+|t|}^{+\infty}\left((\partial_tu_L(t,r))^2+(\partial_ru_L(t,r))^2\right)r^2\,dr\\
 =\int_R^{+\infty} \left( \partial_{r}(ru_0) \right)^2+r^2u_1^2\,dr.
 \end{multline}
\end{prop}
Proposition \ref{P:channels} can be proved directly using the explicit formula for $u_L$, as in the proof of \cite[Lemma 4.2]{DuKeMe13}.

The right-hand side of \eqref{lower_bound} can be compared to the norm in $\HHb$ by a simple integration by parts. Indeed, if $R>0$, we have, for any $f\in \Hb$,
\begin{equation}
\label{IPP}
 \int_{R}^{\infty} \Big|\partial_r(rf(r))\Big|^2dr=\int_{R}^{\infty} \Big|\partial_rf(r)\Big|^2r^2dr-R |f(R)|^2.
\end{equation}
Letting $R\to 0$, we see that the boundary term $R|f(R)|^2$ goes to zero, and \eqref{lower_bound} reads as follows
 \begin{equation}
 \label{lower_bound0}
  \sum_{\pm}\lim_{t\to\pm \infty} \int_{|t|}^{+\infty}|\partial_{t,r}u_L(t,r))|^2
 r^2dr=\int_0^{+\infty} \left(\left|\partial_{r}u_0 \right|^2+|u_1|^2\right)r^2dr.
 \end{equation}
\begin{proof}[Sketch of proof of Theorem \ref{T:rigidity}]
 The proof takes several steps.
\setcounter{step}{0}
\begin{step}
 \label{step:small}
 Let $u$ be as in Theorem \ref{T:rigidity}. Let $\eps>0$ be a small parameter to be specified. In all the proof we fix $R_{\eps}\geq R_0$ such that
 \begin{equation}
  \label{defReps}
  \int_{R_{\eps}}^{+\infty}\left(|\partial_ru_0)|^2+|u_1|^2\right)r^2\,dr\leq \eps^2.
 \end{equation}
 Then
 \begin{equation}
  \label{R4}
  \forall R\geq R_{\eps},\quad
\int_R^{+\infty} \left| \partial_{r}(r u_0) \right|^2+r^2|u_1|^2\,dr\leq CR^5u_0^{10}(R).
 \end{equation}
 This follows from the small data theory (Proposition \ref{P:smalldata}),
 Proposition \ref{P:channels}, the integration by parts formula \eqref{IPP} and the smallness of $\eps$.
\end{step}
\begin{step}
\label{step:limit}
 One can prove that there exists $\theta\in \RR^m$ and $C>0$ such that for large $R$,
 \begin{equation}
  \label{R5}
  \left|u_0(r)-\frac{\theta}{r}\right|\leq \frac{C}{r^3},\quad \int_r^{+\infty}\rho^2|u_1(\rho)|^2\,d\rho\leq \frac{C}{r^5}.
 \end{equation}
 First fix $R$ and $R'$ such that $R_{\eps}\leq R\leq R'\leq 2R$. Letting $\zeta_0(r)=ru_0(r)$, we have, using Cauchy-Schwarz, then Step \ref{step:small}
 \begin{equation}
  \label{R6}
  \left|\zeta_0(R)-\zeta_0(R')\right|\leq \int_{R}^{R'}|\partial_r\zeta_0(r)|\,dr\leq \sqrt{R}\sqrt{\int_R^{R'}|\partial_r \zeta_0|^2dr}\leq \frac{1}{R^2}|\zeta_0|^5(R).
 \end{equation}
The existence of $\theta$ such that \eqref{R5} holds then follows from \eqref{R6}, the fact that $R^{-1}|\zeta_0|(R)^2$ goes to $0$ by radial Sobolev embedding, and elementary real analysis arguments.
\end{step}
\begin{step}
\label{step:ell0}
 In this step, we assume $\theta=0$ and prove that $(u_0,u_1)\equiv (0,0)$. Since by the definition \eqref{defReps} of $R_{\eps}$ and the integration by parts formula \eqref{IPP} one has
 \begin{equation}
  \label{R7}
  \frac{1}{R}\zeta_0^2(R)\leq \eps^2,
 \end{equation}
we deduce from \eqref{R6} that if $R\geq R_{\eps}$ and $k\in \mathbb{N}$,
 $$ \left|\zeta_0\left(2^{k+1}R\right)\right| \geq (1-C\eps^4)\left|\zeta_0\left( 2^kR \right)\right|.$$
 Hence by induction on $k$,
 $$ \left|\zeta_0\left( 2^kR \right)\right| \geq (1-C\eps^4)^{k} |\zeta_0(R)|.$$
 Since by the preceding step and the assumption $\theta=0$, $|\zeta_0(2^kR)|\lesssim 1/(2^kR)^2$, we  deduce, choosing $\eps$ small enough and letting $k\to\infty$ that $\zeta_0(R)=0$. Combining with \eqref{R4} we deduce
 \begin{equation*}
  R\geq R_{\eps}\Longrightarrow \int_{R}^{+\infty} (\partial_r \zeta_0)^2+u_1^2(r)\,dr=0,
 \end{equation*}
 that is $u_0(r)$ and $u_1(r)$ are $0$ for almost every $r\geq R_{\eps}$. Going back to the definition of $R_{\eps}$ we see that we can choose any $R_{\eps}>R_0$, which concludes this step.
\end{step}
\begin{step}
\label{step:Zell}
We next assume $\theta \neq 0$.
We recall
\begin{equation}
 \label{compare}
 \left|Z_{\theta}(r)-\frac{\theta}{r}\right|\lesssim \frac{1}{r^3}
\end{equation}
for large $r$. In this step we prove that $(u_0-Z_{\theta},u_1)$ has compact support. Let $v=u-Z_{\theta}$. Then
\begin{equation}
 \label{eqv}
 \left\{
 \begin{aligned}
\partial_{tt} v-\triangle v&=D_{\lambda}(v)=f(u+v)-f(u).\\
\vec{v}_{\restriction t=0}&=(v_0,v_1):=\left( u_0-Z_{\theta},u_1 \right),
 \end{aligned}
\right.
\end{equation}
For $\eps>0$ small, we fix $R_{\eps}'\gg 1$ such that
\begin{gather}
 \label{defReps1}
 \int_{R_{\eps}'}^{+\infty} \left(|\partial_rv_0(r)|^2+|v_1(r)|^2\right)r^2dr\leq \eps^2\\
 \label{defReps2}
 \int_{\RR} \left( \int_{R_{\eps}'+|t|}^{+\infty} |Z_{\theta}(r)|^{10}r^2\,dr \right)^{\frac 12}dt\leq \eps^{5}.
\end{gather}
Let $v_{L}$ be the solution of \eqref{eq:LW} with initial data $(v_0,v_1)$. Using \eqref{eqv} and the assumptions \eqref{defReps1} and \eqref{defReps2} on $R_{\eps}'$, we obtain
\begin{equation}
\label{f_tf}
\sup_{t\in \RR} \left\|\indic_{\{|x|>|t|+R_{\eps}'\}}\big|\nabla_{t,x}(v(t)-v_L(t))\big|\,\right\|_{\Lb^2}\lesssim \eps^4\Big\|(v_0,v_1)\Big\|_{\HHb(\{r>R'_{\eps}\})}.
\end{equation}
Let $R\geq R_{\eps}'$.
Using Proposition \ref{P:channels},
and the fact that the assumptions of Theorem \ref{T:rigidity} imply
$$\sum_{\pm}\lim_{t\to\pm\infty}\int_{R+|t|}^{+\infty}\left| \partial_{t,r}\big(v(t,r)\big) \right|^2r^2dr=0,$$
we deduce from \eqref{f_tf}
$$\eps^{8}\int_R^{+\infty}\left( |\partial_rv_0|^2+|v_1|^2 \right)r^2dr\gtrsim\int_{R}^{+\infty}\left(\big|\partial_r(rv_0)\big|^2+r^2|v_1|^2\right)dr,$$
and thus
\begin{equation}
\label{boundf0}
\eps^8R|v_0|^2(R)\gtrsim \int_R^{\infty}\left(\big|\partial_r(rv_0)\big|^2+r^2|v_1|^2\right)dr.
\end{equation}
Letting $w_0=rv_0$, we deduce by Cauchy-Schwarz that for $R\geq R_{\eps}'$, $k\in \NN$,
$$\left|w_0\left(2^{k+1}R\right)-w_0\left(2^{k}R\right)\right|\lesssim \int_{2^kR}^{2^{k+1}R} |\partial_r w_0|dr\lesssim \eps^4\left|w_0(2^kR)\right|.$$
This yields by an easy induction $|w_0(2^{k}R)|\geq \left(1-C\eps^4\right)^k|w_0(R)|$, where $C>0$ is a constant which is independent of $\eps$. Since by Step \ref{step:limit},
$$\frac{C}{\left( 2^kR\right)^2}\geq \left|w_0\left( 2^kR \right)\right|,$$
we obtain choosing $\eps$ small enough that $w_0(R)=0$ for large $R$. Combining with \eqref{boundf0}, we deduce that $(f_0(r),f_1(r))=0$ a.e.\,for large $R$, concluding this step.
\end{step}
\begin{step}
\label{step:W}
 In this step we still assume $\ell\neq 0$ and conclude the proof. We let
 $$ \rho=\inf\left\{ R>R_0\;:\; \int_{R}^{+\infty}\left((\partial_r v_0)^2+v_1^2\right)r^2dr=0\right\}$$
 and prove that $\rho=R_0$ i.e.that $u_0(r)=Z_{\theta}(r)$ for $r>R_0$.

 We argue by contradiction, assuming $\rho>R_0$. By the preceding step and finite speed of propagation, the essential support of $v$ is included in $\{r\leq \rho+|t|\}$. Thus $v$ is solution of
\begin{equation*}
 \left\{
 \begin{aligned}
\partial_{tt} v-\triangle v&=\indic_{\{|x|\leq \rho+|t|\}}D_{\lambda}(v).\\
\vec{v}_{\restriction t=0}&=(v_0,v_1):=\left( u_0-Z_{\theta},u_1 \right),
 \end{aligned}
\right.
\end{equation*}
Fix $R_{\eps}''\in (1,\rho)$ such that,
\begin{gather*}
 \int_{R_{\eps}''}^{+\infty} \left(|\partial_rv_0(r)|^2+|v_1(r)|^2\right)r^2dr\leq \eps^2\\
 \int_{\RR} \left( \int_{R_{\eps}''+|t|}^{\rho+|t|} Z_{\theta}^{10}(r)r^2\,dr \right)^{\frac 12}dt\leq \eps^{5}.
\end{gather*}
The same argument as in the preceding step, replacing $R_{\eps}'$ by $R_{\eps}''$, yields that $(v_0,v_1)=0$ for almost every $r>R_{\eps}''$, which contradicts the definition of $\rho$. The proof is complete.
\end{step}
\end{proof}
\begin{corol}
\label{Cor:rigidity}
Let $(u_0,u_1)\in \HHb$. Assume $(u_0,u_1)\notin \left(\ZZZ\cup\{0\}\right)\times \{0\}$. Then there exists $R_0>0$ such that the solution $u$ of \eqref{eq:NLW} with initial data $(u_0,u_1)$ is defined on $\{|x|>R_0+|t|\}$ and satisfies
\begin{gather}
\label{uL5L10_rigi}
u\in \Lb^5\Lb^{10}(\{|x|>R_0+|t|\})\\
 \label{exists_channel}
 \sum_{\pm}\lim_{t\to\pm\infty} \int_{|x|>R_0+|t|}|\nabla_{t,x}u(t,x)|^2dx>0.
\end{gather}
\end{corol}
\begin{proof}
We let $R>0$ such that $\|(u_0,u_1)\|_{\HHR}$ is small. Then by the small data well-posedness theory, the solution $u(t,x)$ is well defined for and in $\Lb^{5}\Lb^{10}$ for $|x|>R+|t|$. If
\begin{equation*}
 \sum_{\pm}\lim_{t\to\pm\infty} \int_{|x|>R+|t|}|\nabla_{t,x}u(t,x)|^2dx>0,
\end{equation*}
we let $R_0=R$ and we are done. If not, by Theorem \ref{T:rigidity}, there exists $\theta\in \RR^m$ such that $(u_0,u_1)=(Z_{\theta},0)$ for $|x|>R$. We let
$$R_1=\min \big\{R>R_{\theta},\; \left\|(u_0,u_1)-(Z_{\theta},0)\right\|_{\HHR}=0\big\},$$
(with $R_{\theta}$ defined in Proposition \ref{P:Ztheta}).
Thus $(u_0,u_1)=(Z_{\theta},0)$ for $\{|x|>R_1\}$. By definition $R_1\geq R_{\theta}$. The equality $R_1=R_{\theta}$ is not possible. Indeed, since $(u_0,u_1)\in \HHb$, it would imply that $R_{\theta}=0$ and $(u_0,u_1)\in (\ZZZ\cup \{0\})\times \{0\}$, contradicting the assumptions of the Corollary.

By finite speed of propagation we obtain that $u(t,r)=Z_{\theta}(r)$ for $r>R_1+|t|$. Since $|Z_{\theta}(r)|\lesssim 1/r$ uniformly for $r>R_1$, it is easy to check that $u\in \Lb^5\Lb^{10}(\{|x|>R+|t|\})$. By Proposition \ref{P:defR-}, if $R_0<R_1$ is close enough to $R_1$, then \eqref{uL5L10_rigi} holds. By Theorem \ref{T:rigidity} and the definition of $R_1$, we also have \eqref{exists_channel}, which concludes the proof.
\end{proof}

\section{Soliton resolution along a sequence of times}
\label{S:sequence_resolution}

In this Section we prove Proposition \ref{P:sequence_resolution} and Theorem \ref{T:sequence_resolution}. As announced in the introduction, we only consider the case $T_+(u)=+\infty$.
\begin{proof}[Proof of Proposition \ref{P:sequence_resolution}]
\emph{Step 1. } We first prove, fixing $A\in \RR$, that
$$u\in \Lb^{5}\Lb^{10}\big(\big\{(t,x),\; t>0,\;|x|>t-A\big\}\big).$$
We let $(t_n)_n$ such that $t_n\to\infty$ and
$$\limsup_{n\to\infty}\|\vec{u}(t_n)\|_{\HHb}<\infty.$$
We will prove that for large $n$,
\begin{equation*}
u\in \Lb^{5}\Lb^{10}\big(\big\{(t,x),\; t>t_n,\;|x|>t-A\big\}\big),
\end{equation*}
i.e.
\begin{equation}
\label{L5L10tn}
u(t_n+\cdot)\in \Lb^{5}\Lb^{10}\big(\big\{(\tau,x),\; \tau>0,\;|x|>\tau+t_n-A\big\}\big).
\end{equation}
We argue by contradiction, assuming that there exists a subsequence of $(t_n)_n$ (still denoted by $(t_n)_n$) such that \eqref{L5L10tn} does not hold for any $n$. Extracting subsequences again, we can assume that $\vec{u}(t_n)$ has a profile decomposition as in Subsection \ref{sub:profile}. With the same notations as in Subsection \ref{sub:profile}, it is sufficient to prove (in view of Theorem \ref{T:approx} with $R_n=t_n-A$) that for all $j\in \JJJ^0$,
\begin{gather}
\label{condition1Uj}
 \lim_{n\to\infty} \frac{t_n}{\lambda_{j,n}}=+\infty \Longrightarrow  U^j\in \Lb^5\Lb^{10}(\{|x|>|t|\})\\
\label{condition2Uj}
 \lim_{n\to\infty} \frac{t_n}{\lambda_{j,n}}=R_0\in (0,\infty)\Longrightarrow \exists \eps>0,\; U^j\in \Lb^5\Lb^{10}(\{|x|>|t|+R_0-\eps\})
\end{gather}
We recall that $\vec{U}^j(0)$ satisfies
$$\vec{U}^j(0)=\wlim_{n\to\infty} \left(\lambda_{j,n}^{1/2} u(t_n,\lambda_{j,n}\cdot), \lambda_{j,n}^{3/2}\partial_tu(t_n,\lambda_{j,n}\cdot)\right).$$
Also, by finite speed of propagation and small data theory,
\begin{equation*}
 \lim_{B\to\infty} \limsup_{t\to\infty}\int_{|x|\geq t+B} |\nabla_{t,x}u(t,x)|^2dx=0
\end{equation*}
This shows that if $\lim_{n\to\infty}\frac{t_n}{\lambda_{j,n}}=+\infty$, then $U^j$ is identically $0$. Hence \eqref{condition1Uj}. Also, if $\lim_{n\to\infty} \frac{t_n}{\lambda_{j,n}}=R_0$, then the support of $\vec{U}^j(0)$ is included in $\{|x|\leq R_0\}$, which shows \eqref{condition2Uj}, concluding Step $1$.

\medskip

\noindent\emph{Step 2. Conclusion of the proof.}
We first fix $A\in \RR$. By Step 1 and the assumption \eqref{H1}, we have $f(u)\indic_{\{|x|\geq t-A\}}\in \Lb^1\Lb^2([0,\infty)\times \RR^3)$. Let $u^A$ be the solution of
\begin{equation}
 \label{NLWA}
 \partial_{tt}u^A-\triangle u^A=f(u)\indic_{|x|>t-A},\quad \vec{u}_{t=0}=(u_0,u_1).
\end{equation}
By finite speed of propagation
\begin{equation}
 \label{finite_speed}
 u(t,x)=u^A(t,x),\quad |x|>t-A.
\end{equation}
By standard energy estimates, one proves that $\vec{S}_L(-t)\vec{u}^A(t)$ satisfies the Cauchy criterion in $\HHb$, and thus has a limit, in $\HHb$, as $t\to\infty$. Thus
there exists a solution $v_L$ of \eqref{eq:LW} such that
\begin{equation}
\label{defvLA}
 \lim_{t\to\infty}\int \left|\nabla_{t,x}(u^A-v_L^A)(t,x)\right|^2dx=0.
\end{equation}

For any $A\in \RR$, we let $g^A$ be the element of $\Lb^2(\RR)$, given by Lemma \ref{L:radiative}, i.e. such that \eqref{radiative} holds with $v_L=v_L^A$ and $g=g^A$. Let $(A,B)\in \RR^2$ with $A>B$. Since by \eqref{finite_speed}, $u^A(t,r)=u^B(t,r)=u(t,r)$ if $A>B$, $r>t-B$, we see that  $g^A(\eta)=g^B(\eta)$ for $\eta<B<A$.Also
\begin{equation}
 \label{bound_g}
 \int_{-\infty}^A (g^A(\eta))^2d\eta\leq M=\liminf_{t\to\infty} \|\partial_tu(t)\|^2_{L^2}
\end{equation}
by \eqref{defvLA}. We define $g(\eta)$ as the common value of all $g^A(\eta)$ with $\eta<A$. By \eqref{bound_g}, we obtain that $g\in L^2(\RR)$. Letting $v_L=S_L(t)(v_0,v_1)$ be the unique solution of \eqref{eq:LW} such that \eqref{radiative} holds, we obtain, by \eqref{finite_speed}, \eqref{defvLA} and the definition of $v_L$ that the conclusion of the proposition holds.
\end{proof}
We next prove Theorem \ref{T:sequence_resolution} in the case $T_+(u)=+\infty$. We let $u$ be a solution of \eqref{eq:NLW}, defined for all $t\geq 0$, with initial data $(u_0,u_1)\in \HHb$ at $t=0$. We consider $t_n\to\infty$ with
$$\lim_{n\to\infty}\|\vec{u}(t_n)\|_{\HHb}<\infty.$$
Extracting subsequences, we can assume that $\vec{u}(t_n)$ has a profile decomposition as in Subsection \ref{sub:profile}. We note also that $v_L$ is (up to transformation) one of the profiles of this profile decomposition, with scaling parameter $1$ and time parameter $-t_n$. Reordering the profiles, we will assume that $v_L$ corresponds to the profile with index $0$, i.e. $U^0_L=v_L$, $\lambda_{0,n}=1$ and $t_{0,n}=-t_n$.

Our goal is to prove that the only nonzero profiles $U^j_{n}$ for $j\geq 1$ satisfy $t_{j,n}=0$ and are stationary solutions of  \eqref{eq:NLW}. The proof is based on the channels of energy method, introduced in \cite{DuKeMe11a}, \cite{DuKeMe12}. For this, we will argue by contradiction, expanding $u$ for $|x|>R_n+|t-t_n|$ (for some $R_n\geq 0$ to be specified) by Theorem \ref{T:approx}, and reaching a contradiction with Corollary \ref{Cor:rigidity}.

By the small data theory and the Pythagorean expansion \eqref{pythagore}, if $j\in \JJJ^0$, then $U^j$ is well defined for $|x|>|t|$ and satisfies $U^j\in \Lb^5\Lb^{10}(\{|x|>|t|\})$, except for a finite number $J_0$ of indices.

\setcounter{step}{0}

\begin{step}
We first prove that $J_0=0$. We argue by contradiction, assuming $J_0\geq 1$. Reordering the profiles, we can assume that for $1\leq j\leq J_0$, $j\in \JJJ^0$ and $U^j$ is not defined for $|x|>|t|$ or satisfies $U^j\notin \Lb^5\Lb^{10}(\{|x|>|t|\})$, and that
\begin{equation}
 \label{order_lambda_jn}
\lambda_{1,n}\ll \ldots \ll \lambda_{J_0,n}.
 \end{equation}
We first observe that $\vec{U}^{J_0}(0)\notin \ZZZ\times \{0\}$. Indeed, from the asymptotics of the elements of $\ZZZ$ (see Corollary \ref{Corol:ZAsymp}), we would obtain $U^{J_0}\in \Lb^{5}\Lb^{10}(\{|x|>|t|\})$, a contradiction with the definition of $J_0$.

By Corollary \ref{Cor:rigidity}, we obtain that there exists $R_0>0$ such that the solution $U^{J_0}$ is well defined for $\{|x|>R_0+|t|\}$, $U^{J_0}\in \Lb^{5}\Lb^{10}(\{|x|>R_0+|t|\})$, and
\begin{equation}
 \label{channel}
 \sum_{\pm}\lim_{t\to\pm\infty} \int_{|x|>R_0+|t|} |\nabla_{t,x}U^{J_0}(t,x)|^2dx=c>0.
\end{equation}
We then apply Theorem \ref{T:approx} with $R_n=R_0\lambda_{J_0,n}$, forward in time and backward in time. By \eqref{order_lambda_jn}, we see that the assumptions of this theorem are satisfied. In the notations of Lemma \ref{L:channel}, we see that $\rho_{J_0,n}=-R_0$, and that \eqref{channel} implies
$$ \int_{-\infty}^{-R_0} \left|g^{J_0}_+(\eta)\right|^2d\eta\geq \frac{c}{2}\quad\text{or}\quad\int_{-\infty}^{-R_0} \left|g^{J_0}_-(\eta)\right|^2d\eta\geq \frac{c}{2},$$
where $g^{J_0}_{\pm}$ is by definition such that
\begin{multline*}
 \lim_{t\to\pm \infty}\int_{0}^{\infty} \big|\partial_tU^{J_0}(t,r)-\frac{1}{r}g_{\pm}^{J_0}(\pm t-r)\big|^2r^2dr\\
 =\lim_{t\to\infty}\int_{0}^{\infty} \big|\partial_rU^{J_0}(t,r)+\frac{1}{r}g_{\pm}^{J_0}(\pm t-r)\big|^2r^2dr=0,
\end{multline*}
Using \eqref{limE_n}, we see that for large $n$, the following bound holds for one of the signs $+$ or $-$
$$\lim_{\tau \to\pm \infty} \int_{|x|>\lambda_{J_0,n}R_0+|\tau|} \left|\nabla_{t,x}\big(u(t_n+\tau,x)-v^L(t_n+\tau,x)\big)\right|^2dx\geq \frac{c}{2}.$$
This is a contradiction with the definition of $v^L$ if the sign is $+$, or with the fact that
$$ \lim_{t\to-\infty}\int_{|x|\geq A+|t|} |\nabla_{t,x}(u-v_L)|^2dx\underset{A\to\infty}{\longrightarrow} 0$$
if the sign is $-$.
\end{step}
\begin{step}
\label{St:zero_prof}
 Next, we prove that for all $j\geq 1$, we have $U^j_L\equiv 0$, or $j\in \JJJ^0$, there exists $Z_j \in \ZZZ$ such that $(U^j,0)=(Z_j,0)$. Indeed, if not, we would combine Theorem \ref{T:approx} with $R_n=0$, Lemma \ref{L:channel} and Corollary \ref{Cor:rigidity}, and argue in the preceding step to prove that there exists $c'>0$ such that for large $n$,
$$\sum_{\pm}\lim_{\tau \to\pm \infty} \int_{|x|>|\tau|} \left|\nabla_{t,x}\big(u(t_n+\tau,x)-v^L(t_n+\tau,x)\big)\right|^2dx\geq c'.$$
 \end{step}
\begin{step}
 By Step \ref{St:zero_prof}, there is only a finite number $J_1$ of nonzero profile in the expansion of $\vec{u}(t_n)$. Thus $w^J_n$ is independent of $J$ for $J\geq J_1$. Using the same argument as in Step \ref{St:zero_prof}, together with \eqref{lower_bound0}, we obtain
 $$ \lim_{n\to\infty} \left\|\vec{w}^{J_1}_n(0)\right\|_{\HHb}=0.$$
 This concludes the proof.
\end{step}

\section{Boundedness of the energy norm of the solution for potential-type nonlinearities}
\label{S:bounded}

In this Section, we prove Proposition \ref{P:bounded}. We will assume that the nonlinearity is potential-like, i.e. that it satisfies Assumption \ref{H0} in the introduction. Note that since the potential $F(u)$ is homogeneous of degree $6$, we have
\begin{equation}
u\cdot f(u)=u \cdot (\nabla_u F)(u) = 6 F(u)
\label{Eqn:HomFGrad}
\end{equation}
We start by proving that for potential-like nonlinearities, global solutions are bounded along a sequence of times.
\begin{lem}
\label{L:boundedness}
Let $(u_{0},u_{1}) \in \HHb $. Let $u$ be a solution of (\ref{eq:NLW}) with data $(u_{0},u_{1})$ at $t=0$. Assume that $T_{+}(u) = +\infty$. Then $E(u) \geq 0$ and
\begin{equation}
\liminf_{t \rightarrow \infty} \left\| \nabla u(t) \right\|^{2}_{L^{2}} + \left\| \partial_{t} u(t) \right\|^{2}_{L^{2}} \leq 3E(u) \cdot
\nonumber
\end{equation}
\end{lem}

\begin{proof}

We adapt the proof of the analogous result in the scalar case in \cite{DuKeMe13}, which is a modification of the proof of blow-up for negative energy solutions in \cite{Levine74}. 

Assume by contradiction that this is not true. Then either  $(i)$: $E(u) < 0 $  holds or $(ii)$: $ E (u) \geq 0 $ and there exists
$ 1 \gg \epsilon_{0} > 0 $ such that, for $t \gg 1$,
\begin{equation}
\begin{array}{l}
\left\| \nabla u(t) \right\|^{2}_{L^{2}} + \left\| \partial_{t} u(t) \right\|^{2}_{L^{2}} > (3+\epsilon_0)E (u_{0},u_{1}) + \epsilon_{0}
\end{array}
\label{Eqn:HypContr}
\end{equation}
Let $y(t) := \int_{\mathbb{R}^{3}} \phi  \left( \frac{x}{t} \right) |u(t,x)|^{2} \,dx$ with $\phi$ a radial, smooth, and compactly supported function such
that $\phi(x) = 1 $ if $|x| \leq 2$ and $\phi(x) =0$ if $|x| \geq 3$. We have

\begin{equation}
\begin{array}{l}
y^{'}(t) = A_{1} + A_{2} = \int_{\mathbb{R}^{3}} \langle u, \partial_{t} u \rangle \phi \left( \frac{x}{t} \right) \,dx - \frac{1}{t^{2}}
\int_{\mathbb{R}^{3}} |u|^{2} x \cdot \nabla \phi \left( \frac{x}{t} \right) \,dx
\end{array}
\nonumber
\end{equation}
Our goal is to control the second term. \\
\\
\underline{Claim}: Let $\epsilon > 0 $. Then there exists $ R \gg 1$ such that for all $t\geq 0$,
\begin{equation}
\int_{|x| \geq R + t}  | \nabla u(t,x)|^{2} + | \partial_{t} u(t,x)|^{2} \,dx \leq \epsilon \cdot
\label{Eqn:EstOutCone}
\end{equation}
Indeed let $R \gg 1$ be such that $X \ll \epsilon^{2} $ with $X := \int_{|x| \geq R} | \nabla u_{0} |^{2} + |u_{1}|^{2} \;  dx  $. Let $\tilde{u}$ be the solution of the wave equation $\partial_{tt} \tilde{u} - \triangle \tilde{u} = f(\tilde{u})$ with initial data $ \left( \tilde{u}(0), \partial \tilde{u}(0) \right) := \Psi_{R}(u_{0},u_{1})$,
with $\Psi_{R}$ defined in the proof of Corollary \ref{Cor:UniquenessElliptic}.  From $ X = \left\| \left( \tilde{u}(0), \partial_t \tilde{u}(0) \right)
\right\|_{\HHb} \ll \epsilon $, the small data theory of (\ref{eq:NLW}) applied to $\tilde{u}$, and the finite speed of propagation we get (\ref{Eqn:EstOutCone}). \\
\\
From the Hardy inequality outside the ball $B \left( O , 2 |t| \right)$ and the claim above we get $A_{2} = o(t)$. Hence

\begin{equation}
\begin{array}{l}
| y^{'}(t) | \leq 2 \int_{|x| \le 2 t} |u(t)| |\partial_{t} u(t)| \,dx + o(t) \cdot
\end{array}
\label{Eqn:Estypr}
\end{equation}
Integration by parts, the equality (\ref{Eqn:HomFGrad}) and the same argument as above to control the remainder terms show that

\begin{equation*}
y^{''}(t) = 2 \int_{\mathbb{R}^{3}} | \partial_{t} u(t)|^{2} - |\nabla u(t) |^{2} \,dx  + 12 \int_{\mathbb{R}^{3}} F(u) \,dx + o(1)
\end{equation*}
Now from the conservation of energy we get

\begin{equation}
y^{''}(t) = 8 \| \partial_{t} u \|^{2}_{L^{2}} + 4 \| \nabla u(t) \|^{2}_{L^{2}} - 12  E(u) +o(1).
\label{Eqn:EqDbpry}
\end{equation}
Assume that $(ii)$ holds. Then the above estimate and (\ref{Eqn:HypContr}) imply that for large $t$
\begin{equation}
\begin{array}{l}
y^{''}(t) \geq\left( 4 +\epsilon_0 \right) \| \partial_{t} u \|^{2}_{L^{2}}  + \epsilon_0.
\end{array}
\label{Eqn:EstDbpry}
\end{equation}
Assume now that $(i)$ holds. Then it is clear from (\ref{Eqn:EqDbpry}) that (\ref{Eqn:EstDbpry}) also holds. Hence by integration there exists $\alpha > 0 $ such that
$y^{'}(t) \geq \alpha t$ for all $ t \gg 1$. As a consequence, using also (\ref{Eqn:Estypr}) and  the Cauchy-Schwarz inequality we get that for $t\gg 1$,
\begin{equation}
\begin{array}{ll}
y^{'}(t) & \leq (2+\eps_0/100) \| u(t) \|_{L^{2} (|x| \leq 2 t)}
 \| \partial_{t} u(t) \|_{L^{2}(|x| \leq 2 t)}  \\
& \leq \left( 2 +\epsilon_0/100 \right) y^{\frac{1}{2}}(t) \left( \frac{y^{''}(t)}{4 +\epsilon_0} \right)^{\frac{1}{2}}
\end{array}
\nonumber
\end{equation}
So there exists $ \gamma > 1 $ such that  $\gamma (y^{'})^{2} \leq y y^{''}$ for $t \gg 1$. This implies that
$\frac{d}{dt} \log \left( \frac{y^{'}}{y^{\gamma}} \right) (t) \geq 0 $. Hence choosing $t_{1} \gg 1 $ we get
$\frac{y^{'}}{y^{\gamma}}(t) \geq \frac{y^{'}}{y^{\gamma}}(t_{1}) > 0 $ for $t \geq t_{1}$.
Hence  $ \frac{d}{dt} \left( \frac{1}{y^{\gamma -1}} \right)(t) \leq  (1 - \gamma ) \frac{y^{'}(t_{1})}{y^{\gamma}(t_{1})}<0  $ for $t \geq t_{1}$: this contradicts
the non-negativity of $y(t)$ for $t \gg 1$.
\end{proof}

We next prove Proposition \ref{P:bounded} as a corollary of Lemma \ref{L:boundedness} and Theorem \ref{T:sequence_resolution}. We assume again \eqref{H0}.
We let $v_L$ be as in the conclusion of Theorem \ref{T:sequence_resolution}, and $(v_0,v_1)$ be the initial data of $v_L$. We claim that for any sequence $t_n\to\infty$:
\begin{equation}
\label{bound_explicit}
 \limsup_{n\to\infty}\|\vec{u}(t_n)\|_{\HHb}<\infty\Longrightarrow \limsup_{n\to\infty}\|\vec{u}(t_n)\|_{\HHb}^2\leq 3E(u).
\end{equation}
Assuming \eqref{bound_explicit}, we see that Lemma \ref{L:boundedness} implies the desired conclusion
$$ \limsup_{t\to\infty}\|\vec{u}(t)\|_{\HHb}^2\leq 3E(u),$$
by a straightforward contradiction argument using the intermediate value theorem. To prove \eqref{bound_explicit}, we argue by contradiction, assuming that there exists a sequence $t_n\to\infty$ such that
\begin{equation}
\label{contradiction_explicit}
 \limsup_{n\to\infty}\|\vec{u}(t_n)\|_{\HHb}<\infty\text{ and }\limsup_{n\to\infty}\|\vec{u}(t_n)\|_{\HHb}^2>3E(u).
\end{equation}
By Theorem \ref{T:sequence_resolution}, there exist a subsequence of $(t_n)_n$ (that we still denote by $(t_n)_n$), $J\geq 0$, and, for all $j\in \{1,\ldots,J\}$, a sequence of positive numbers $(\lambda_{j,n})_{n}$ and an element $Q_j$ of $\ZZZ$ such that \eqref{Eqn:OrthCondLambda} and \eqref{expansion_seq} hold. As a consequence of these two properties, we obtain (arguing as in Remark \ref{R:energy})
\begin{equation}
 \label{expansion1}
 E(u)=\sum_{j=1}^JE(Q_j)+\frac{1}{2}\|(v_0,v_1)\|^2_{\HHb}=
\frac{1}{3}\sum_{j=1}^J\|Q_j\|^2_{\Hb}+\frac{1}{2}\|(v_0,v_1)\|^2_{\HHb},
\end{equation}
where we have used the identity \eqref{I:Z1} of Proposition \ref{Prop:PropZ}. By a similar argument, we obtain
\begin{equation}
 \label{expansion2}
 \lim_{n\to\infty} \|\vec{u}(t_n)\|^2_{\HHb}= \sum_{j=1}^J\|Q_j\|^2_{\Hb}+\|(v_0,v_1)\|^2_{\HHb}.
\end{equation}
Combining \eqref{expansion1} and \eqref{expansion2}, we obtain
$$\lim_{n\to\infty} \|\vec{u}(t_n)\|^2_{\HHb}\leq 3E(u),$$
contradicting \eqref{contradiction_explicit} and concluding the proof of Proposition \ref{P:bounded}.
\qed

\section{Continuous-in-time soliton resolution}
\label{S:continuous_resolution}

In this section we prove the continuous-in-time
resolution into stationary solutions. We assume \eqref{H0}, \eqref{Eqn:AssFinite} and \eqref{Eqn:AssEj} throughout the section.

 Recall that by the Pohozaev identity, one has that $\int_{\mathbb{R}^{3}} |\nabla Q|^2 =6 \int_{\mathbb{R}^{3}} F(Q)$ for all $Q\in \mathcal{Z}$. Thus \eqref{H0} and (\ref{Eqn:HomFGrad}) imply
 that
 \begin{equation}
  \label{SR10}
  \forall Q \in \ZZZ, \quad \frac{1}{3} \int_{\mathbb{R}^{3}} |\nabla Q|^2 \,dx = 2 \int_{\mathbb{R}^{3}} F(Q) \,dx = E(Q).
 \end{equation}
 We define

\begin{equation}
\label{SR11}
K = \left\{ Q \in \ZZZ,\quad \int_{|x| \leq 1}  |\nabla Q|^{2} \,dx =  \frac{3 E(Q)}{2} \right\}.
\end{equation}

We first prove that $K$ is compact (\S \ref{sub:K}). In \S \ref{sub:compact}, we state and prove a simple preliminary but crucial result on compact sets.
In \S \ref{sub:SameOrder}, we prove that the number of profiles and the energy of the profiles is the same in all asymptotic decompositions of the form \eqref{expansion_seq}.

 \subsection{Compactness of $K$}
\label{sub:K}

In this subsection, we prove that:

\begin{prop}

Assume \eqref{H0},\eqref{Eqn:AssFinite} and \eqref{Eqn:AssEj}. Then $K$, defined by (\ref{SR11}),  is compact in $\Hb$.

\label{Prop:CompactnessK}
\end{prop}

The proof of Proposition \ref{Prop:CompactnessK} relies upon the lemma below:

\begin{lem}

\label{Lem:CompactnessK}
Assume \eqref{H0},\eqref{Eqn:AssFinite} and \eqref{Eqn:AssEj}.
Let $( Q_{n} )_{n \geq 1}$ be a sequence of elements of $\ZZZ$. Then, after extraction of a subsequence from $ (Q_n )_{n \geq 1}$, there exists
$ Z  \in  \mathcal{Z}$, $ \{ \lambda_{n} \} $, and  $ w_{n} \in  \Hb $ satisfying for $n \gg 1$
\begin{equation}
Q_{n} = \frac{1}{\lambda_{n}^{\frac{1}{2}}} Z \left( \frac{x}{\lambda_{n}} \right) + w_{n}\text{ where }
\lim_{n \rightarrow \infty} \| \nabla w_{n} \|^{2}_{L^{2}} = 0.
\label{Eqn:DecompSameNrj}
\end{equation}
\end{lem}

The compactness of $K$ follows from Lemma \ref{Lem:CompactnessK}. Indeed, let $ ( Q_n )_{n \geq 1}$ be a sequence in $K$. Then (after extraction of a subsequence) we see
that (\ref{Eqn:DecompSameNrj}) holds. Extracting again subsequences, we can assume
$ \lim \limits_{n \rightarrow \infty} \lambda_n = \ell \in [0,\infty] $. Hence
\begin{equation}
\label{EQ_n}
\frac 32 E(Q_n)=\int_{|x|\leq 1}|\nabla Q_n(x)|^2dx=\int_{|x|\leq \ell}|\nabla Z(x)|^2dx+o(1).
\end{equation}
By the Assumption \eqref{Eqn:AssFinite}, we conclude that $E(Q_n)$  is constant for large $n$. Thus $E(Q_n)=E(Z)=\frac{1}{3}\int |\nabla Z|^2$. This shows that $Z$ is not identically $0$ and also, by \eqref{EQ_n}, that $\ell\in (0,\infty)$. Rescaling $Z$ we can assume that $\ell=1$, and thus $Q_{n} \rightarrow Z$ in $\Hb$, which shows, since $\mathcal{Z}$ is closed in $\Hb$ (Proposition \ref{Prop:PropZ} (\ref{I:Z3})), that $Z\in \ZZZ$. By \eqref{EQ_n} we conclude that $Z\in K$, which proves as announced that$K$ is compact in $\Hb$.

\begin{proof}[Proof of Lemma \ref{Lem:CompactnessK} ]
Let $ ( Q_{n})_{n \geq 1}$ be a sequence in $\mathcal{Z}$ . Observe that it is bounded: this follows
from (\ref{SR10}) and (\ref{Eqn:AssFinite}). Hence, by the profile decomposition of \cite{Gerard98}, there exist, up to a subsequence, $J_0 \in \NN \cup \{\infty\}$,
 $\{ Z_{j} \}_{1 \leq j <J_0 +1}$ and $ \{ \lambda_{j,n} \}_{ 1\leq j <J_{0}+1}$ such that $ Z_{j} \neq 0 $ , $\lambda_{j,n} > 0$,
and, for all $ 1 \leq J < J_{0} + 1 $   we have

\begin{align}
\label{Eqn:EqQn}
Q_{n} &= \sum \limits_{j=1}^{J} \frac{1}{\lambda_{j,n}^{\frac{1}{2}}} Z_{j} \left( \frac{x}{\lambda_{j,n}} \right) + w_{n}^{J}, \\
\label{Eqn:Pythagore}
\| \nabla Q_{n} \|^{2}_{L^{2}} &= \sum \limits_{j=1}^{J} \| \nabla Z_{j} \|^{2}_{L^{2}} + \| \nabla w_{n}^{J} \|^{2}_{L^{2}}+o_n(1) ,
\end{align}
$ Z_{j} = \wlim_{n\to\infty} \, \lambda_{j,n}^{\frac{1}{2}} Q_{n} \left( \lambda_{j,n} x \right) $, in $\Hb$, with
$$j\neq k\Longrightarrow \lim_{n\to\infty}\frac{\lambda_{j,n}}{\lambda_{k,n}} \in \{0,\infty\}$$
(orthogonality of scaling parameters), $ 1 \leq j \leq J: \;
\wlim_{n \to \infty} \lambda_{j,n}^{\frac{1}{2}} w_{n}^{J}(\lambda_{j,n} \cdot) = 0  $,  and $ \lim \limits_{J \rightarrow J_0} \limsup \limits_{n \rightarrow \infty} \| w_{n}^{J} \|_{L^{6}} =0 $. 

Since $\mathcal{Z}\cup\{0\}$ is closed for the weak topology of $\Hb$, (see Proposition \ref{Prop:PropZ}, (\ref{I:Z3})), we see that $ Z_{j} \in \mathcal{Z} $, and there exists $c_{0} > 0 $ such that
$ \| \nabla Z_{j} \|^{2}_{L^{2}} \geq c_{0}$. Hence (\ref{Eqn:Pythagore})  shows that  $ 0 \leq  J_0  < \infty $. As a consequence,
\begin{equation}
\begin{array}{l}
Q_{n} = \sum \limits_{j=1}^{J_{0}} \frac{1}{\lambda^{\frac{1}{2}}_{j,n}} Z_{j} \left( \frac{x}{\lambda_{j,n}} \right) + w_{n},
\end{array}
\nonumber
\end{equation}
with $ \lim \limits_{n \rightarrow \infty} \| w_{n} \|_{\Lb^{6}} = 0$. Observe that $J_{0} \neq 0$: if not,
$ \lim_{n \to \infty} \| Q_n\|_{\Lb^6}=0$, which would contradict the Pohozaev identity \eqref{SR10}. The elementary estimate
\begin{equation}
\begin{array}{l}
 \left| F \left( \sum \limits_{j=1}^{\bar{j}} z_{j}  \right) - \sum \limits_{j=1}^{\bar{j}} F(z_{j}) \right| \lesssim
 \sum \limits_{j \neq j'} |z_{j}| |z_{j'}|^{5},
\end{array}
\label{Eqn:EstFDecoupl}
\end{equation}
show that there exist two constants $C_{1} > 0$ and $C_{2} > 0$ such that
\begin{multline*}
\left|\int_{\mathbb{R}^{3}}  F(Q_{n}) \,dx -
\sum \limits_{j=1}^{J_{0}}  \int_{\mathbb{R}^{3}}  F(Z_{j}) \,dx - \int_{\mathbb{R}^{3}} F(w_{n}) \,dx
\right|\\
\leq C_{1}  \sum \limits_{j \neq j'} \int_{\mathbb{R}^{3}} \frac{1}{\lambda_{j,n}^{\frac{1}{2}}} \left| Z_{j} \left( \frac{x}{\lambda_{j,n}} \right) \right|
\frac{1}{\lambda_{j',n}^{\frac{1}{2}}}  \left| Z_{j'} \left( \frac{x}{\lambda_{j',n}} \right) \right|^{5} \,dx\\
+  C_{2}  \sum \limits_{j=1}^{J_{0}} \int_{\mathbb{R}^{3}} | w_{n}(x) |^{5} \frac{1}{\lambda_{j,n}^{\frac{1}{2}}} \left| Z_{j} \left( \frac{x}{\lambda_{j,n}} \right) \right| \,dx
\end{multline*}
We infer that each term of the second line goes to zero as $n \rightarrow \infty$. Indeed an approximation argument in $L^{6}$ shows that we may assume without loss of generality
that the $Z_{j}$ are smooth and compactly supported; next, from (\ref{Eqn:OrthCondLambda}), elementary changes of variables, and H\"older inequality, we
can prove that the statement holds. It is also clear from H\"older inequality and the smallness property of the $L^{6}-$ norm of $w_{n}$ for $n \gg 1$ that
each term of the third line goes to zero as $n \rightarrow \infty$. 

Using also (\ref{SR10}) we get $ \frac{E(Q_{n})}{2} = \sum \limits_{j=1}^{J_{0}} \frac{E(Z_{j})}{2} + o_{n}(1) $. By \eqref{Eqn:AssFinite} we can assume that $E(Q_n)$ is constant, which yields $ \frac{E(Q_{n})}{2} = \sum \limits_{j=1}^{J_{0}} \frac{E(Z_{j})}{2}$.
Hence, combining with (\ref{Eqn:AssEj}) we get
$J_{0}= 1$ and $E(Q_{n}) = E(Z_{1})$ for all $n$. Moreover (\ref{Eqn:Pythagore}) implies that
$ \lim \limits_{n \rightarrow \infty} \| \nabla w_{n} \|^{2}_{L^{2}} = 0$.

%
\end{proof}

\subsection{An elementary property of compact sets}
\label{sub:compact}
We next prove:
\begin{lem}
\label{L:compact}
 Let $(E,d)$ be a metric space and $K$ a compact subset of $E$. Let $T_0\in \RR$. Les $U:[T_0,\infty)\to E$ such that for all sequence $(t_n)_{n\in \NN}$, there exists a subsequence $(t_n)_{n\in I}$ and $\Phi\in K$ such that
 $$ \lim_{n\to\infty} d(U(t_n),\Phi)=0.$$
 Then there exists $U_K:[T_0,\infty)\to K$ such that
 $$ \lim_{t\to\infty} d(U(t),U_K(t))=0.$$
\end{lem}

\begin{proof}
For $X\in E$, we denote $d(X,K)=\inf_{Y\in K} d(X,Y)$. Since $K$ is compact, there exists $X_K\in K$ such that $d(X,K)=d(X,X_K)$. Furthermore the assumptions of the claim and a straightforward contradiction argument shows that $\lim_{t\to\infty} d(U(t),K)=0$. Choosing $U_K(t)$ such that
$$ d(U(t),U_K(t))=d(U(t),K),$$
we obtain the conclusion of the claim.
\end{proof}

\subsection{Conservation of the energies of the asymptotic states}
\label{sub:SameOrder}

In this subsection we assume \eqref{H0}, \eqref{Eqn:AssFinite} and \eqref{Eqn:AssEj}, and consider a solution $u$ with $T_+(u)=+\infty$. We will also assume that $u$ does not scatter to a linear solution as $t\to\infty$. By Proposition \ref{P:bounded}, we can use Proposition \ref{P:sequence_resolution}. We let $v_L$ be given by the conclusion of this proposition. 
We prove the following:
\begin{prop}
Let $t^{\pm}_{n} \rightarrow \infty $ be two sequences of times such that there exist $ J^{\pm} \in  \mathbb{N}$, positive numbers
$ \lambda^{\pm}_{1,n}  \ll \lambda^{\pm}_{2,n} \ll ... \ll \lambda^{\pm}_{J^{\pm},n}  $, and $ Q^{\pm}_{j} \in \mathcal{Z}$ with energy $ E(Q^{\pm}_{j})$ for $ 1 \leq j \leq J^{\pm} $ for which
\begin{equation}
\begin{array}{l}
\vec{u}(t^{\pm}_{n}) = \vec{v}_{L}(t^{\pm}_{n}) + \sum \limits_{j=1}^{J^{\pm}} \frac{1}{\left( \lambda_{j,n}^{\pm} \right)^{\frac{1}{2}}}
\left( Q^{\pm}_{j} \left( \frac{x}{\lambda^{\pm}_{j,n}} \right) , 0 \right)
+ o_{n}(1)\text{ in }\HHb.
\end{array}
\label{Eqn:ProfDecompJPlusMin}
\end{equation}
Then $J^{+} = J^{-} $, and $E(Q^{+}_{j}) = E(Q^{-}_{j})$ for $ 1 \leq j \leq J $, with $ J := J^{+} =J^{-} $.

\label{Prop:SameOrderNrj}
\end{prop}

The proof relies upon the lemma below:

\begin{lem}
Let $ 0 <  \epsilon \ll 1  $. There exists $ \bar{C} > 0 $ such that the following holds. Let $ t_{n} \rightarrow \infty $ such that one can find $J\in \mathbb{N}$, and,  for $1 \leq j \leq J $, $Q_{j} \in \mathcal{Z}$ with energy $E(Q_{j})$, and $(\lambda_{j,n})_{n} $ for which
(\ref{Eqn:OrthCondLambda}) and (\ref{expansion_seq}) hold. Fix $j$ with $1\leq j\leq J$. Let 
$$ \tilde{\lambda}_{j} (t) := \inf \left\{ \lambda : \;  \int_{|x| \leq \lambda} \left| \nabla \left(  u(t) - v_{L}(t )\right) \right|^{2} + e^{-|x|- |t|} \,dx
-  3 \sum \limits_{k=1}^{j-1} E(Q_{k})  \geq \epsilon \right\} $$ 
and  
$$ \widetilde{E}_{j}(t) := \int_{|x| \leq \bar{C} \tilde{\lambda}_{j} (t) } | \nabla (u - v_{L})(t)|^{2} \,dx
- 3 \sum \limits_{k=1}^{j-1} E(Q_{k}).$$
Then 
\begin{equation}
\limsup_{n\to\infty} \left| \widetilde{E}_{j}(t_{n}) - 3 E(Q_{j}) \right| \leq \epsilon.
\label{Eqn:EsttildeEj}
\end{equation}
\label{Lem:EsttildeEj}
\end{lem}
\begin{rem}
\label{R:tlambdaj}
 Since $\lambda\mapsto \int_{|x| \leq \lambda} \left| \nabla \left(  u(t) - v_{L}(t )\right) \right|^{2} + e^{-|x|- |t|} \,dx$ is increasing, $\tlambda_j(t)$ is the only $\lambda$ such that 
 $$ \int_{|x| \leq \lambda} \left| \nabla \left(  u(t) - v_{L}(t )\right) \right|^{2} + e^{-|x|- |t|} \,dx=\eps$$
\end{rem}
We postpone the proof of 
Lemma \ref{Lem:EsttildeEj} and prove Proposition \ref{Prop:SameOrderNrj}.

\begin{proof}
Let $ 0 <  \epsilon \ll 1  $ be small enough such that all the statement below are true. We may assume without loss of generality that $J_{-} \leq  J_{+}$.
We proceed by contradiction, assuming that there exists $J'\in \NN$ with 
$ 0\leq  J' \leq J_{-} $ such that:  
\begin{gather*}
1\leq  j \leq J' \Longrightarrow E(Q^{+}_{j}) = E(Q^{-}_{j})\\
J'=J_-<J_+\text{ or } \Big(J'<J_-\leq J_+\text{ and }E(Q_{J'+1}^-)\neq E(Q_{J'+1}^+)\Big).
\end{gather*}
Extracting subsequences and reindexing $(t_n^+)_n$ if necessary, we will also assume $t_n^-<t_n^+$ for all $n$.

We apply Lemma \ref{Lem:EsttildeEj} to the sequences $(t_n^-)_n$ (with $Q_j=Q_j^-$) and $(t_n^+)_n$ (with $Q_j=Q_j^+$). We denote by $\tilde{\lambda}_j^{\pm}(t)$, $\widetilde{E}_j^{\pm}(t)$ the quantities defined in this lemma in each of these two cases. Note that by the definition of $J'$, we have $\tE_j^+(t)=\tE_j^-(t)$ and $\tlambda_j^+(t)=\tlambda_j^-(t)$ for $0\leq j\leq \min(J_-,J'+1)$.

We let 
$$m_1=\min_{Q\in \mathcal{Z}} E(Q), \quad m_2=\min_{\substack{(Q,R)\in \mathcal{Z}^2\\ E(Q)\neq E(R)}} \Big|E(Q)-E(R)\Big|, \quad m=\min(m_1,m_2)>0.$$

\medskip

\noindent\emph{Case 1: $J'=J_{-} < J_{+}$.}
Let 
$$\delta(t)=\sum_{j=1}^{J'} \left|\tE_j^-(t)-3E(Q_j^-)\right|+|\tE_{J'+1}^+(t)|,$$

Since $Q_k^+=Q_k^-$ for $k=1,\ldots,J'$, we have
$$ \tE_{J'+1}^+(t_n^-)= \int_{|x| \leq \bar{C} \tilde{\lambda}_{J'+1}^+ (t_n^-) } | \nabla (u - v_{L})(t_n^-)|^{2} \,dx
- 3 \sum \limits_{k=1}^{J'} E(Q_{k}^-).$$
Obviously, $\tlambda_{J'}^-(t_n^-)=\tlambda_{J'}^+(t_n^-)\leq \tlambda_{J'+1}^+(t_n^-)$. Thus 
\begin{equation}
\label{sandwich}
\tE^-_{J'}(t_n^-)-3E(Q_{J'}^-)\leq \tE^+_{J'+1}(t_n^-)\leq \int_{\RR^3} | \nabla (u - v_{L})(t_n^-)|^{2} \,dx
- 3 \sum \limits_{k=1}^{J_-} E(Q_{k}^-).
\end{equation} 
Noting that the right-hand side of \eqref{sandwich} goes to $0$ as $n$ goes to infinity, and using Lemma \ref{Lem:EsttildeEj} to bound the left-hand side, we obtain
$$\limsup_{n}|\tE^+_{J'+1}(t_n^-)|\leq \epsilon.$$
Using Lemma \ref{Lem:EsttildeEj} to bound the other terms in the definition of $\delta(t_n^-)$ we deduce
\begin{equation}
\label{limdeltatn-}
\limsup_{n\to\infty} \delta(t_n^-)\leq (J'+1)\epsilon.
\end{equation} 
Furthermore by Lemma \ref{Lem:EsttildeEj} applied to the sequence $(t_n^+)_n$, 
\begin{equation}
\label{limdeltatn+}
\limsup_{n\to\infty} \delta(t_n^+)\geq\limsup_{n\to\infty} \left|\tE^+_{J'+1}(t_n^+)\right|\geq m-\epsilon. 
\end{equation} 
Observe that $ t \rightarrow \delta(t)$ is a continuous function:
this follows from the continuity of $\tilde{\lambda}_{j}$, $1\leq j\leq J'$, and $\tilde{\lambda}_{J'+1}^+$
(the proof is similar to that of the continuity of $\lambda_{j}$  in end of the proof of Theorem \ref{T:continuous_resolution} below, see the argument after \eqref{SR44Sev}) and that of $ t \rightarrow u(t) - v_{L}(t)$ in $\Hb$. By the intermediate value theorem, there exists $t_n\in (t_n^-,t_n^+)$ such that
\begin{equation}
\label{deft_ndelta}
\delta(t_n)=(2J'+100)\epsilon. 
\end{equation} 
By Theorem \ref{T:sequence_resolution}, extracting subsequences, there exists $J\geq 1$, and for all $j\in \{1,\ldots, J\}$, $(\lambda_{j,n})_n$ and $Q_j\in \mathcal{Z}$ such that \eqref{Eqn:OrthCondLambda} and \eqref{expansion_seq} hold for the sequence $(t_n)_n$.  We prove by finite induction on $j$:
\begin{equation}
\label{FiniteInduction}
 \forall j\in \{1,\ldots,J'\}, \quad j\leq J\quad \text{and} \quad E(Q_j)=E(Q_j^-)=E(Q_j^+).
\end{equation} 
Note that the equality $E(Q_j^+)=E(Q_j^-)$ follows from the definition of $J'$.

If $J'=0$, there is nothing to prove. Assuming $J'\geq 1$, we will prove simultaneously the case $j=1$ and the induction step.

Let $k$ with $1\leq k\leq J'$, and assume $E(Q_j)=E(Q_j^-)$ for $1\leq j\leq k-1$ (we do not make any assumption if $k=1$). This implies $\tE_j=\tE_j^-$ for $1\leq j\leq k$. We will prove $E(Q_{k})=E(Q_{k}^-)$. 

By the definition of $\delta(t)$ and $t_n$, 
$$ \left|\tE_k^-(t_n)-3E(Q_k^-)\right|\leq (2J'+100)\epsilon.$$
Furthermore, by Lemma \ref{Lem:EsttildeEj}, and since $\tE_k(t_n)=\tE_k^-(t_n)$, we have, for large $n$,  $\left|\tE_k^-(t_n)-3E(Q_k)\right|\leq 2\epsilon$. Hence $\left|3E(Q_k^-)-3E(Q_k)\right|\leq (2J'+102)\epsilon$. Taking $\epsilon$ small enough, we obtain $\left|3E(Q_k^-)-3E(Q_k)\right|<m$, which implies $E(Q_k^-)=E(Q_k)$, concluding the proof of \eqref{FiniteInduction}.

Let $j\in \{1,\ldots, J'\}$. Using \eqref{FiniteInduction} (and the fact that it implies $\tE_j=\tE_j^-$), Lemma \ref{Lem:EsttildeEj} yields $|\tE_j^-(t_n)-3E(Q_j^-)|=|\tE_j(t_n)-3E(Q_j)|\leq 2\epsilon$ for large $n$. 

Using also \eqref{deft_ndelta} and the smallness of $\epsilon$, we obtain that for large $n$
$$100\eps\leq |\tE_{J'+1}^+(t_n)|\leq (2J'+100)\epsilon\leq \frac{m}{10}.$$
This is impossible, since by Lemma \ref{Lem:EsttildeEj}, and \eqref{FiniteInduction} we have
$$ \limsup_{n\to\infty}| \tE_{J'+1}^+(t_n)|\begin{cases}
  \leq \epsilon&\text{ if }J'=J\\
  \geq m-\epsilon&\text{ if }J'<J.
                                         \end{cases}
$$
\medskip

\noindent\emph{Case 2: $J'<J_{-} \leq J_{+}$.}
The argument is almost the same as the one in case $1$, modifying the definition of $\delta(t)$ to 
$$\delta(t)=\sum_{j=1}^{J'+1} \left|\tE_j^-(t)-3E(Q_j^-)\right|.$$
Using that by the definitions of $J^-$ and $m$, $|E(Q_{J'+1}^-)-E(Q_{J'+1}^+)|\geq m$  we see that \eqref{limdeltatn-} and \eqref{limdeltatn+} hold. For large $n$, by the intermediate value theorem,  we thus can find $t_n\in (t_n^-,t_n^+)$ such that \eqref{deft_ndelta} holds. 

We use again Theorem \ref{T:sequence_resolution}, extracting subsequences, to obtain $J\geq 0$, and for all $j\in \{1,\ldots, J\}$, $(\lambda_{j,n})_n$ and $Q_j\in \mathcal{Z}$ such that \eqref{Eqn:OrthCondLambda} and \eqref{expansion_seq} hold. With the same proof than \eqref{FiniteInduction}, we obtain
\begin{equation*}
 \forall j\in \{1,\ldots,J'+1\}, \quad j\leq J\quad \text{and} \quad E(Q_j)=E(Q_j^-)=E(Q_j^+).
\end{equation*} 
This implies by Lemma \ref{Lem:EsttildeEj}
$$(2J'+100)\epsilon=\delta(t_n)=\limsup_{n\to\infty} \delta(t_n)\leq 2(J'+1)\epsilon,$$
a contradiction.


\end{proof}

It remains to prove Lemma \ref{Lem:EsttildeEj}. The proof is based upon the claim below:
\begin{claim}
Let $ \epsilon > 0 $. Then there exists $ \tC > 0 $ such that for all $\eta > 0$ and for
all $ Q \in \mathcal{Z} $ satisfying $ \int_{|x| \leq \eta} |\nabla Q (x)|^{2} \,dx  \geq
\epsilon $ we have
\begin{equation}
\int_{|x| \leq \tC \eta} |\nabla Q (x)|^{2} \,dx \geq 3 E(Q) - \epsilon \cdot
\nonumber
\end{equation}
\label{Claim:GrothLocNorm}
\end{claim}

\begin{proof}
Rescaling, we may assume that $\eta =1$. 

We argue by contradiction. Assuming that the conclusion does not hold, there exists a sequence $( Q_{n} )_{n \geq 1} $  such that
$ Q_{n} \neq 0 $,  $ \int_{|x| \leq 1} | \nabla Q_{n}(x)|^{2} \,dx \geq  \epsilon $ , and
$ \int_{|x| \leq n }  |\nabla Q_{n}|^{2} \,dx \leq 3 E(Q_{n}) - \epsilon $.  Let $\lambda_{n}$ be such that $\tQ_{n} \in K$
(recall that $K$ is defined in (\ref{SR11}) ) with $\tQ_{n}$ defined by
$ Q_{n}(x) = \frac{1}{\lambda_{n}^{\frac{1}{2}}} \tQ_{n}  \left( \frac{x}{\lambda_{n}} \right) $. Extracting a
subsequence if necessary there exists $ \tQ \in K$ such that $ \tQ_{n} \rightarrow \tQ $ in $\Hb$. Hence, a change of
variable yields  
$$ \int_{|x| \leq  \lambda^{-1}_{n} }  | \nabla \tQ(x) |^{2} \, dx  \geq \epsilon \text{ and }\int_{|x| \leq n \lambda_{n}^{-1}} | \nabla \tQ(x) |^{2} \, dx  \leq \int |\nabla \tQ(x)|^2\,dx- \epsilon.$$
This is a contradiction, since the first inequality implies that $\lambda_n^{-1}$ is bounded from below by a stricty positive constant, which in turn implies $\lim_{n\to\infty}n\lambda_n^{-1}=+\infty$, contradicting the second inequality.


\end{proof}

\begin{proof}[Proof of Lemma \ref{Lem:EsttildeEj}]

We note that for all $j$, $c\mapsto \int_{|x|\leq c} |\nabla Q_j(x)|^2dx$ is increasing. Indeed, the zeroes of the function $r\mapsto |\nabla Q_j(r)|^2$ are isolated because of uniqueness in Cauchy-Lipschitz theorem.

Let $ c_{j} > 0 $, be such that $ \int_{|x| \leq c_{j}} |\nabla  Q_{j}(x)|^{2} \,dx =\epsilon $. We first prove 
\begin{equation}
\label{limit}
\lim \limits_{n \rightarrow \infty} \frac{\tilde{\lambda}_{j}(t_{n})}{\lambda_{j,n}} = c_j  
\end{equation} 

Indeed, using elementary changes of variables and (\ref{Eqn:OrthCondLambda}), we have, for $\bar{c}>0$
\begin{equation}
\int_{|x| \leq \bar{c} \lambda_{j,n} } \left|  \nabla \left( u(t) - v_{L}(t) \right) \right|^{2} \,dx
= 3 \sum \limits_{k=1}^{j-1} E(Q_{k}) +  \int_{|x| \leq \bar{c}}  \left| \nabla Q_{j}(x) \right|^{2} \,dx + o(1).
\nonumber
\end{equation}
This shows:
$$ \lim_{n\to\infty}\int_{|x|\leq \bar{c}\lambda_{j,n}}|\nabla (u(t)-v_L(t))|^2dx-3\sum_{k=1}^{j-1}E(Q_k)\begin{cases} >\epsilon&\text{ if }\bar{c}>c_j\\ <\epsilon &\text{ if }\bar{c} <c_j.\end{cases}$$
Letting $c_-<c<c_+$, we obtain $c_{-} \lambda_{j,n} \leq \tilde{\lambda}_{j}(t_{n})\leq c_{+} \lambda_{j,n}$ for large $n$, which shows the claimed limit. 

Let $\tC$ be given by Claim \ref{Claim:GrothLocNorm}. Using that by \eqref{limit},
$$\lim_{n\to\infty} \tC\tlambda_{j}(t_n)/\lambda_{k,n}=\begin{cases} +\infty&\text{ if }k<j\\ 0&\text{ if }k>j\end{cases},$$
we obtain
\begin{multline*}
\int_{|x| \leq \tC\lambda_j(t_n)} \left|  \nabla \left( u(t) - v_{L}(t) \right) \right|^{2} \,dx\\
= 3 \sum \limits_{k=1}^{j-1} E(Q_{k}) +  \int_{|x| \leq \tC\tilde{\lambda}_j(t_n)/\lambda_{j,n} }  \left| \nabla Q_{j}(x) \right|^{2} \,dx + o(1)\\
= 3 \sum \limits_{k=1}^{j-1} E(Q_{k}) +  \int_{|x| \leq \tC c_j}  \left| \nabla Q_{j}(x) \right|^{2} \,dx + o(1),
\nonumber
\end{multline*}
where the last equality follows also from \eqref{limit}.
By Claim \ref{Claim:GrothLocNorm}
$$3E(Q_j)-\epsilon\leq \int_{|x| \leq \tC c_j}  \left| \nabla Q_{j}(x) \right|^{2} \,dx \leq 3E(Q_j),$$
and the conclusion of the Lemma follows.
\end{proof}

\subsection{End of the proof of the resolution into stationary solutions}

\label{sub:end_of_proof}

In this subsection we prove Theorem \ref{T:continuous_resolution}. 

\smallskip

\noindent\emph{Step 1.} We prove that $\vec{u}(t)$ is bounded in $\HHb$, and more precisely that there exists a
$ ( J'_{0}, E'_{1},..., E_{J'_{0}} ) \in \mathbb{N} \times E(\mathcal{Z} )^{J_{0}'}$ such that
\begin{equation}
 \label{SR40Mult}
 \lim_{t \to \infty} \|\vec{u}(t)-\vec{v_L}(t)\|_{\HHb}^2 = 3  \sum \limits_{j=1}^{J^{'}_{0}} E^{'}_{j} \cdot
\end{equation}

and, for all sequence $t_n\to +\infty$ such that $\eqref{Eqn:OrthCondLambda}$ and \eqref{expansion_seq} hold, we have 
\begin{equation}
 \label{SR40'Mult}
 J=J'_0\quad \text{and}\quad \forall j\in 1,\ldots, J'_0 \text{ and }E(Q_j)=E'_j.
\end{equation}

Indeed, first observe by Proposition \ref{P:bounded} that there exists a sequence $ t_n' \to \infty$ such that $\|\vec{u}(t_n')\|_{\HHb}$ is bounded. Hence,
by Theorem \ref{T:sequence_resolution} (more precisely, by (\ref{expansion_seq}), taking into account (\ref{Eqn:OrthCondLambda}) and (\ref{Eqn:PohozaevEq})), we see that there exist $ J'_{0} < \infty $, and $Q'_{1}$,..., $ Q'_{J^{'}_{0}} \in \mathcal{Z}$ with energy $E(Q'_{1})$,..., $E(Q'_{J^{'}_{0}})$ respectively such that
\begin{equation}
\lim_{n \to \infty} \| \vec{u}(t_{n}')- \vec{v_L}(t_{n}') \|_{\HHb}^2 = 3 \sum \limits_{j=1}^{J'_{0}} E(Q'_{j}) \cdot
\label{Eqn:IntermLimuvL}
\end{equation}

We let $E_j'=E(Q_j')$ and prove \eqref{SR40Mult} by contradiction. Assuming that (\ref{SR40Mult}) does not hold, there exist $ t_{n} \rightarrow \infty $ such that
\begin{equation*}
\lim_{t_{n} \to \infty} \| \vec{u}(t_{n})- \vec{v_L}(t_{n}) \|_{\HHb}^2 \neq 3 \sum \limits_{j=1}^{\bar{J}} E'_j\cdot
\end{equation*}

Hence an application of the intermediate value theorem shows that there exists $t^{''}_{n} \rightarrow \infty $ such that there exists $ \eta > 0 $ such that
\begin{equation}
\lim_{t^{''}_{n} \to \infty} \|\vec{u}(t^{''}_{n})- \vec{v_L}(t^{''}_{n}) \|_{\HHb}^2 = 3 \sum \limits_{j=1}^{J^{'}_{0}} E'_j+ \eta \cdot
\nonumber
\end{equation}
This implies that $ \{  \vec{u}(t^{''}_{n}) \} $ is a bounded sequence in $\HHb $.  Hence we can apply Theorem \ref{T:sequence_resolution}
and then Proposition \ref{Prop:SameOrderNrj} combined with (\ref{Eqn:IntermLimuvL}) shows that in fact
\begin{equation}
\lim_{t^{''}_{n} \to \infty} \| \vec{u}(t^{''}_{n})- \vec{v_L}(t^{''}_{n}) \|_{\HHb}^2 = 3 \sum \limits_{j=1}^{J^{'}_{0}} E(Q_{j}') \cdot
\nonumber
\end{equation}
This is clearly not possible. 
The equality \eqref{SR40'Mult} then follows easily from Proposition \ref{Prop:SameOrderNrj}.

\medskip

\noindent\emph{Step 2. Construction of the scaling parameters.}
Letting $(J'_{0}, E^{'}_{1}, E^{'}_{2},...,E^{'}_{J_{0}^{'}}) \subset \mathbb{N} \times  E(\mathcal{Z})^{J_{O}^{'}} $ as in Step $2$, we define, for $1 \leq j \leq J^{'}_0$
\begin{equation}
\label{SR44Sev}
\lambda_j(t)= \inf \left\{ \lambda > 0, \; \int_{ |x| \leq \lambda} \left(|\nabla(u(t)-v_L(t))|^{2}+ e^{-|x|-|t|} \right) dx  \geq
3 \sum \limits_{k=1}^{j-1} E^{'}_{k} + \frac{3}{2} E^{'}_{j}
\right\}.
\end{equation}
By \eqref{SR40Mult}, $\lambda_j(t)$ is well-defined, and a positive number for all $ t \geq T_0 $, where $T_0$ is large. In this step, we prove $\lambda_j\in \mathcal{C}^0([T_0,\infty),[0,\infty))$. For this we will use the term $e^{-|x|-|t|}$ which appears in \eqref{SR44Sev}. 

We fix $T\geq T_0$ and prove that $\lambda_j(t)$ is continuous at $t=T$. If not, there exists $T_n\to T$ and $\ell\in [0,\infty]$ with $\ell\neq \lambda_j(T)$ such that 
$$\lim_{n\to\infty}\lambda_j(T_n)=\ell.$$
By the definition of $\lambda_j(t)$, we have, for $t\geq T_0$,
$$\int_{|x|\leq \lambda_j(t)}\left(|\nabla(u(t)-v_L(t))|^2+e^{-|x|-|t|}\right)dx = 3 \sum \limits_{k=1}^{j-1} E^{'}_{k} + \frac{3}{2} E^{'}_{j} .$$
Letting $t=T_n$, $n\to\infty$ in the preceding, we obtain 
\begin{multline*}
\int_{|x|< \ell}\left(|\nabla(u(T)-v_L(T))|^2+e^{-|x|-|T|}\right)dx = 3 \sum \limits_{k=1}^{j-1} E^{'}_{k} + \frac{3}{2} E^{'}_{j}\\ =\int_{|x|< \lambda_j(T)}\left(|\nabla(u(T)-v_L(T))|^2+e^{-|x|-|T|}\right)dx.
\end{multline*}
Hence $\ell=\lambda_j(T)$, a contradiction.

\medskip

\noindent\emph{Step 3. Construction of $Q_j(t)$.}
We denote
$$ U_j(t,x)=\lambda_j^{1/2}(t) (u-v_L)(t,\lambda_j(t)x).$$
In this step we show the existence of $ Q_j(t) \in  K $, for $ 1 \leq j \leq J^{'}_0$, $t\geq T_0$ ($T_0$ large) such that
\begin{equation}
 \label{PropertyQj}
 U_j(t)-Q_j(t)\xrightharpoonup{t\to\infty} 0,\quad\text{weakly in } \Hb.
\end{equation}
Let $(e_k)_{k\in \NN}$ be a Hilbert basis of $ \Hb $. If $\varphi$ and $\psi$ are in $ \Hb $, we let
$$d(\varphi,\psi)=\sum_{k \in \NN} \frac{1}{2^k}\left|\langle \varphi-\psi, e_k \rangle\right|.$$
Then it is easy to check that $d$ is a metric on $ \Hb $, and that
$ \lim_{n \to \infty}d ( \varphi_n, \psi) = 0 $ in
$$ \mathcal{B}=\left\{ \varphi \in \Hb ,\; \|\varphi\|^{2}_{ \Hb }\leq 1+ 3 \sum \limits_{j=1}^{J^{'}_{0}} E^{'}_{j} \right\} $$
if and only if $ \varphi_n \xrightharpoonup{n \to \infty} \psi \text{ weakly in } \Hb$. 
We also note that $ K \subset \mathcal{B} $, and that it is compact for the topology induced by $d$ (since it is compact for the strong topology of $\Hb$, as proved above).

By Step 1, $U_j(t)\in \mathcal{B}$ for $t$ large enough.
By Lemma \ref{L:compact}, the following property will imply
\eqref{PropertyQj}: 
\begin{equation}
 \label{lim_d_Uj_K}
 \lim_{t\to\infty} d(U_j(t),K)=0.
\end{equation} 

Let $t_n\to\infty$. By Theorem \ref{T:sequence_resolution} and Step 1 (rescaling the $\Phi^j$ to be in $K$), extracting subsequences, there exist elements $\Phi^1,\ldots,\Phi^{J^{'}_0}$ of $K$ such that
\begin{gather}
\label{SR70}
 \vec{u}(t_n)-\vec{v}_L(t_n)=\sum_{k=1}^{J^{'}_0}\left( \frac{1}{\lambda_{k,n}^{1/2}} \Phi^k\left( \frac{\cdot}{\lambda_{k,n}} \right),0 \right)+o_n(1)\text{ in }\HHb\\
 \label{SR70'}
 \lambda_{1,n}\ll \ldots\ll\lambda_{J_0',n}.,
\end{gather}
with $\forall j\in \{1,\ldots,J_0'\}$, $3E(\Phi^j)=\int |\nabla \Phi_j|^2=E'_j$.
Fixing $j \in \{1,\ldots,J_0^{'} \}$, we claim
\begin{equation}
 \label{SR71}
 \lim_{n\to\infty}\frac{\lambda_{j,n}}{\lambda_j(t_n)}=1.
\end{equation}

Indeed, by \eqref{SR70} and \eqref{SR70'}, 
\begin{multline}
 \label{SR72}
 \int_{|x|\leq c\lambda_{j,n}}|\nabla(u(t_n,x)-v_L(t_n,x)|^2dx \\
 = \sum_{k=1}^{j-1}  \int \left|\nabla \Phi^k(x)\right|^2 \,dx +
  \int_{|x| \leq c} |\nabla \Phi^j(x)|^2dx  + o_n(1).
\end{multline}
Using that $\Phi^j\in K$, we obtain $\int_{|x|\leq 1}\Phi_j=\frac{3}{2}E_j'$. Thus \eqref{SR72} and the definition of $\lambda_j(t)$ implies 
that for any $\delta>0$, we have that for large $n$,
$$(1-\delta)\lambda_{j,n}\leq \lambda_{j}(t_n)\leq (1+\delta)\lambda_{j,n},$$
which yields \eqref{SR71}.

By \eqref{SR70}, \eqref{SR70'}, we have
$$ \lambda_{j,n}^{1/2}(u-v_{L} )(t_n,\lambda_{j,n}\cdot)\xrightharpoonup{n\to\infty} \Phi^j \text{ weakly in } \Hb.$$
Combining with \eqref{SR71}, we deduce
$$ \lim_{n\to\infty} d(U_j(t_n),\Phi^j)=0,$$
which implies \eqref{lim_d_Uj_K}, concluding this Step 3.


\noindent

\emph{Step 4.} We prove that the expansion \eqref{SR12} holds, with $J_{0} = J_{0}^{'}$ and the $\lambda_j(t)$, $Q^j(t)$ constructed in the preceding steps. We argue by contradiction, assuming that there exists $\eps>0$ and $t_n\to\infty$ such that for all $n$,
\begin{equation}
 \label{no_expansion}
 \left\|\vec{u}(t_n) -\vec{v}_{L}(t_n) - \sum_{j=1}^{J^{'}_0} \left(\frac{1}{\lambda_{j}^{\frac{1}{2}}(t_n)} Q^{j} \left( t_n, \frac{\cdot}{\lambda_{j}(t_n)} \right),0\right)\right\|_{\HHb}\geq \eps.
 \end{equation}
Using Theorem \ref{T:sequence_resolution}  and Proposition \ref{Prop:CompactnessK} (rescaling the solitons if necessary to be in $K$), we obtain (extracting a subsequence from $\{t_n\}$ if necessary) $(\Phi^1,\ldots,\Phi^{J_0'})\in K^{J_0}$ and $(\lambda_{1,n},\ldots,\lambda_{J_0',n})$ such that \eqref{SR70} and \eqref{SR70'} hold. Note that Step 1 ensures that the number of profiles is exactly $J_0'$. Arguing as in Step 3, we obtain
\begin{gather}
\label{SR80}
\lambda_{j,n}^{1/2}(u-v_{L})( t_n,\lambda_{j,n}\cdot)\xrightharpoonup{n\to\infty}\Phi^j\text{ weakly in } \Hb \cdot \\
\label{SR81}
\lim_{n\to\infty}\frac{\lambda_j(t_n)}{\lambda_{j,n}}=1
\end{gather}
By the construction of $Q^j$ in step 3, we also have
\begin{equation}
\label{SR82}
\lambda_{j}^{1/2}(t_n)(u-v_{L})(t_n,\lambda_{j}(t_n)\cdot)-Q^j(t_n)\xrightharpoonup{n\to\infty}0\text{ weakly in } \Hb.
\end{equation}
Since $K$ is a compact subset of $ \Hb $, extracting subsequences, we can assume that there exists $\widetilde{\Phi}^j$ such that
\begin{equation}
 \label{SR83}
 \lim_{n\to\infty}\left\|Q^{j}(t_n)-\widetilde{\Phi}^j\right\|_{\Hb}=0.
\end{equation}
Combining \eqref{SR80},\ldots,\eqref{SR83}, we see that $\widetilde{\Phi}^j=\Phi^j$. Using \eqref{SR81} and \eqref{SR83} again, we see that the expansion \eqref{SR70} implies
$$\lim_{n\to\infty}
 \left\|\vec{u}(t_n) -\vec{v}_{L}(t_n) - \sum_{j=1}^{J_{0}^{'}} \left(\frac{1}{\lambda_{j}^{\frac{1}{2}}(t_n)} Q^{j} \left( t_n, \frac{\cdot}{\lambda_{j}(t_n)} \right),0\right)\right\|_{\HHb}=0,$$
 contradicting \eqref{no_expansion}. The proof is complete.
\qed

\section{Higher dimensions}

\label{S:higherD}

In this section we discuss the generalization to higher dimensions of the results above, i.e. the equation \eqref{eq:NLW} on $\RR^N$, $N\geq 4$, where
$$ \HHb =\dot{H}^1_{rad}(\RR^N,\RR^m)\times L^2_{rad}(\RR^N,\RR^m),$$
and $f\in \mathcal{C}^2(\RR^m\setminus \{0\},\RR^m)$ is homogeneous of degree $\frac{N+2}{N-2}$ (note that due to the homogeneity, it is not possible to have nonlinearities that are $\mathcal{C}^2$ at the origin if $N\geq 7$).

The major obstruction to the adaptation of the proof of Theorems \ref{T:sequence_resolution}, \ref{T:continuous_resolution} and Theorem \ref{T:rigidity} to higher dimension is the fact that the lower bound of the exterior energy given by Proposition \ref{P:channels} of solutions of linear wave equation must be replaced by a weaker analog in dimension $N\geq 4$: see \cite{CoKeSc14}, \cite{KeLaLiSc15}, \cite{DuKeMaMe22}, \cite{LiShenWei24}.

Nevertheless, in dimension $4$, for the scalar equation $\partial_{tt}u-\triangle u=u^3$, the analogs of the continuous soliton resolution, Theorem \ref{T:continuous_resolution}, and the rigidity theorem, Theorem \ref{T:rigidity} are valid, see \cite{DuKeMaMe22}. We conjecture that these theorems will remain valid for a general system of equation of the form \eqref{eq:NLW} with $f$ homogeneous of degree $3$. We note however that the proof of \cite{DuKeMaMe22} does not adapt completely to the case of systems, since it uses arguments on the sign of the solution which are of course not valid for systems.

For the scalar equation
\begin{equation}
\label{NLW_scalar}
\partial_{tt}u-\triangle u=|u|^{\frac{4}{N-2}}u
\end{equation}
with $N\geq 5$, the analog of the strong rigidity theorem \ref{T:rigidity} is not valid (see \cite{CoDuKeMe23a}), and only a weaker version of the continuous soliton resolution is known, with the additional assumption:
\begin{equation}
\label{bound}
\sup_{t\geq 0}\|(u,\partial_tu)(t)\|_{\dot{H}^1\times L^2}<\infty.
\end{equation}
(see \cite{DuKeMe23} for $N$ odd, \cite{CoDuKeMe22Pb} for $N=6$ and \cite{JendrejLawrie23} for any $N\geq 5$).
The following weaker rigidity result is proved in these works:
\begin{thm}
\label{T:weak_rigidity}
 Assume that $N\geq 5$. Let $u$ be a radial global solution of \eqref{NLW_scalar} which is bounded in the energy space. Assume
 \begin{equation}
  \label{no_channel_at_all}
  \forall A\in \RR,\quad \lim_{t\to\pm\infty}\int_{|x|>|t|+A}|\nabla_{t,x}u(t,x)|^2dx=0.
 \end{equation}
 Then $u$ is a stationary solution.
\end{thm}
To illustrate the difficulty of proving the results in the present paper in high dimensions, we give an example of a system of the form \eqref{eq:NLW}, where $f$ is homogeneous of degree $\frac{N+2}{N-2}$, with $m=2$ and $N\geq 5$ such that the analog of Theorem \ref{T:rigidity} is not valid, and the soliton resolution does not hold in full generality. Consider
\begin{equation}
\label{NLW_example}
 \left\{
 \begin{aligned}
  \partial_{tt}\varphi-\triangle \varphi&=|\varphi|^{\frac{4}{N-2}}\varphi\\
  \partial_{tt}\psi-\triangle \psi&=|\varphi|^{\frac{4}{N-2}}\psi.
 \end{aligned}
 \right.
\end{equation}
Let $W=\left(1+\frac{|x|}{N(N-2)}\right)^{\frac{2-N}{2}}$ be the ground state stationary solution of \eqref{NLW_scalar}. Then $\Phi=(\varphi,\psi)=(W,tW)$ is a global, finite energy solution of \eqref{NLW_example} such that
$$\forall A\in \RR,\quad \lim_{t\to\infty}\int_{|x|>|t|+A}|\nabla_{t,x}\Phi(t,x)|^2dx=0$$
and
\begin{equation}
\label{}
\lim_{t\to\pm\infty}\int_{\RR^N} |\nabla \Phi(t,x)|^2dx=+\infty.
\end{equation}
This shows that an analog of the rigidity theorem, Theorem \ref{T:rigidity} is false in dimension $N\geq 5$ without further assumption on the nonlinearity or on the solution.
The solution $\Phi$ is also a counter-example to the strong soliton resolution conjecture, asserting that any solution such that $T_+=\infty$ decomposes into rescaled stationary solution.

We do not know however if the conclusion of the rigidity theorem or the soliton resolution (for a sequence of times or for continuous times) hold for solutions that are assumed to be bounded in the energy space.

We note also that the nonlinearity $f$ in \eqref{NLW_example} does not derive from a potential. For nonlinearities of the form $f(u)=\nabla_uF(u)$, with $F$ homogeneous of degree $\frac{2N}{N-2}$, it is easy to adapt the proof of Proposition \ref{P:bounded} to prove that a global solution is bounded along a sequence of times. We conjecture that the analog of the weak rigidity theorem \ref{T:weak_rigidity} and the soliton resolution results \ref{T:sequence_resolution} and \ref{T:continuous_resolution} are valid with the  additional assumption that $f$ derives from a potential, however the adaptation of the works on scalar equations in this setting seems to be far from being trivial.

\end{document}